\documentclass[12pt]{amsart}
\usepackage{amsfonts}
\usepackage{amsfonts,latexsym,rawfonts,amsmath,amssymb,amsthm}
\usepackage[plainpages=false]{hyperref}

\usepackage{graphicx}

\RequirePackage{color}

\numberwithin{equation}{section}

\newcommand{\beq}{\begin{equation}}
\newcommand{\eeq}{\end{equation}}
\newcommand{\beqs}{\begin{eqnarray*}}
\newcommand{\eeqs}{\end{eqnarray*}}
\newcommand{\beqn}{\begin{eqnarray}}
\newcommand{\eeqn}{\end{eqnarray}}
\newcommand{\beqa}{\begin{array}}
\newcommand{\eeqa}{\end{array}}

\newtheorem{prop}{Proposition}[section]
\newtheorem{theo}[prop]{Theorem}
\newtheorem{lem}[prop]{Lemma}

\newtheorem{cor}[prop]{Corollary}
\newtheorem{rem}[prop]{Remark}

\renewcommand{\div}{\mbox{div}\,}

\allowdisplaybreaks
\title{Interior estimates for the Monge-Amp\`ere type \\ fourth order equations}
\author{Ling Wang and Bin  $\text{Zhou}^*$}

\address{School of Mathematical Sciences, Peking
University, Beijing 100871, China.}

\email{lingwang@stu.pku.edu.cn; bzhou@pku.edu.cn}

\thanks {*This research is partially supported by National Key R$\&$D Program of China SQ2020YFA0712800 and NSFC  grants 11822101.}

\begin{document}
\subjclass[2020]{35J30, 35J96, 35B45, 35B65}
\keywords{Monge-Amp\`ere equation, linearized Monge-Amp\`ere equation, fourth order equation, partial Legendre transform, degenerate elliptic equations}
\begin{abstract}
In this paper, we give several new approaches to study interior estimates for  a class of fourth order equations of Monge-Amp\`ere type. First, we prove interior estimates for the homogeneous equation in dimension two by using the  partial Legendre transform. As an application, we obtain a new proof of the Bernstein theorem without using  Caffarelli-Guti\'errez's estimate, including the Chern conjecture on affine maximal surfaces. For the inhomogeneous equation,  we also obtain a new proof in dimension two by an integral method relying on the Monge-Amp\`ere Sobolev inequality.  This proof works even when the right hand side is singular. In higher dimensions, we obtain the interior regularity  in terms of integral bounds on  the second derivatives and the inverse of the determinant.
\end{abstract}

\maketitle

%\bibliographystyle{plain}
%\tableofcontents

\baselineskip=16.4pt
\parskip=3pt

\section{Introduction}

We study the regularity of the following fourth order equations of
Monge-Amp\`ere type
\beq\label{4-eq}
U^{ij}w_{ij}=f,
\eeq
where $\{U^{ij}\}$ is the cofactor matrix of $D^2u$ of an unknown uniformly convex function,
and
\beq\label{w-d}
w=\begin{cases}
\ [\det D^2u]^{-(1-\theta)}, & \theta\geq 0, \ \ \theta\neq 1,\\[4pt]
\ \log\det D^2u, & \theta=1.
\end{cases}
\eeq
When $\theta=\frac{1}{n+2}$, this is the {\it affine mean curvature equation} in affine geometry \cite{Ch}.
When $\theta=0$, it is {\it Abreu's equation} arising from the problem of extremal metrics on toric manifolds in K\"ahler geometry \cite{Ab}, and is equivalent to
$$\sum_{i,j}\frac{\partial^2 u^{ij}}{\partial x_i\partial x_j}=f,$$
where $\{u^{ij}\}$ is the inverse matrix of $D^2u$. The regularity of \eqref{4-eq} has been extensively studied before, see \cite{TW1,TW2, D, Z1, Z2, CHLS, Le1, Le2, CW}. This equation is usually treated as a system of a Monge-Amp\`ere equation and a linearized Monge-Amp\`ere equation. Therefore,
in previous works, its regularity relies heavily on Caffarelli-Guti\'errez's  deep result on the interior regularity of the linearized Monge-Amp\`ere equation  \cite{CG}, which was later extended by \cite{LS, GN1, GN2} to the the boundary and to higher order estimates. In this paper,  we investigate the interior estimates of \eqref{4-eq} by
several new approaches.
We will mainly concentrate on the case $\theta\in[0,1]$ due to the interesting geometric background.

We first consider the case of the homogeneous equation
\beq\label{4-eq-h}
U^{ij}w_{ij}=0,
\eeq
where $w$ is given by \eqref{w-d}.
We apply  the partial Legendre transform to give a new proof of the interior estimates of \eqref{4-eq-h} in dimension two.

\begin{theo}\label{int-est}
Assume $n=2$ and  $\theta\in [0,1]$.
 Let $\Omega\subset\mathbb R^2$ be a convex domain and let
$u$ be a  smooth convex solution to equation \eqref{4-eq-h} on $\Omega$ satisfying
\beq\label{detcond}
0<\lambda\leq\det D^2u\leq \Lambda.
\eeq
Then for any $\Omega'\Subset\Omega$, there exists a constant $C>0$ depending on $\sup_\Omega|u|$, $\lambda$, $\Lambda$, $\theta$ and $dist(\Omega',\partial\Omega)$, such that
$$\|u\|_{C^{4, \alpha}(\Omega')}\leq C.$$
\end{theo}

The partial Legendre transform for the fourth order equation was first used in \cite{LZ} where the authors deal with the second boundary value problem. After the partial Legendre transform, equation \eqref{4-eq-h} becomes a quasi-linear second order equation (see \eqref{eq-new1}) for the determinant.
The main ingredient in our proof is an interior integral gradient estimate (Theorem \ref{grad}).
When $\theta\in \left[0, \frac{1}{4}\right]$, condition \eqref{detcond} holds by the determinant estimates and arguments of strict convexity \cite{TW1, Z1}. By Theorem \ref{int-est} and a rescaling argument as in \cite[Theorem 2.1]{TW1}, we obtain a new proof of the following Bernstein theorem \cite{TW1, JL, Z1} without using Caffarelli-Guti\'errez's theory.
\begin{theo}
Assume $n=2$ and $0\leq \theta\leq\frac{1}{4}$.
Let $u$ be an entire smooth uniformly convex solution to \eqref{4-eq-h} on $\mathbb R^2$.
Then $u$ is a quadratic polynomial.
\end{theo}

In the case of the inhomogeneous equation and in higher dimensions, the partial Legendre transform does not work.
We will investigate the interior regularity by an integral method motivated by De Giorgi-Nash-Moser's theory.
Consider the  inhomogeneous equation with general right hand side term
\beq\label{4-eq-inh}
U^{ij}w_{ij}=f(x, u, Du, D^2u),
\eeq
where $w$ given by \eqref{w-d} with $\theta\in[0,1]$.
This equation is introduced by \cite{Le4, Le5} in the study of convex functionals with a convexity constraint related to the Rochet-Chon\'e model for the monopolist problem in economics.
It is said to be singular since the right hand side term depends on $D^2u$.
A typical example considered in \cite{Le4, Le5, LZ}  is
\beq\label{plap}
f=\div(|\nabla u|^{p-2}\nabla u) + f^0(x, u).
\eeq
Note that when $f\in L^\infty(\Omega)$,  once we have the determinant estimate \eqref{detcond},
we can use Caffarelli-Guti\'errez's theory to get the interior regularity. The assumption on $f$ can be weakened to $f\in L^{\frac{n}{2}+\varepsilon}(\Omega)$ by \cite{LN}. Hence for $f$ defined as \eqref{plap}, when $n=2$ and $p\geq 2$, we can obtain the interior regularity of $u$ directly by using interior $W^{2,1+\varepsilon}$-estimates of the Monge-Amp\`ere equations \cite{DFS, S}. To settle the more  singular case $1<p<2$,  Le  established
the interior estimate of the linearized Monge-Amp\`ere equation with right hand side term in divergence form in dimension two \cite{Le3}.
One of the main tools in \cite{Le3} are the  Monge-Amp\`ere Sobolev inequality (see Lemma \ref{Sobolev}). In this paper, we will
use the Monge-Amp\`ere Sobolev inequality and the $W^{2,1+\varepsilon}$-estimates for the Monge-Amp\`ere equation directly in the fourth order equation to obtain a $W^{2,p}$-estimate of the solution (Theorem \ref{W2p-est}).
Then we can apply the regularity theory of second order elliptic equation of  divergence type  to obtain a new proof for the interior estimates of \eqref{4-eq-inh}.

\begin{theo}\label{int-est-inh}
Assume $n=2$ and  $\theta\in [0,1]$.   Let $\Omega\subset\mathbb R^2$ be a convex domain. Assume $$f=\div(g)+h$$
where $g:=\left(g^1(x), g^2(x)\right):\Omega\to\mathbb R^2$ is a bounded vector function and $h\in L^{q}(\Omega)$ for some $q>\frac{n}{2}$.
Suppose
$u$ is a smooth convex solution to equation \eqref{4-eq-inh} on $\Omega$ satisfying  \eqref{detcond}.
Then for any $\Omega'\Subset\Omega$, there exists a constant $C>0$ depending on $\sup_\Omega|u|$, $\lambda$, $\Lambda$, $\theta$, $\|g\|_{L^{\infty}(\Omega)}$, $\|h\|_{L^{q}(\Omega)}$ and $dist(\Omega',\partial\Omega)$, such that
$$\|u\|_{W^{4, q}(\Omega')}\leq C.$$
\end{theo}

\begin{rem}
(1) It is clear that the above theorem applies to the case \eqref{plap} for any $p>1$ in dimension two. The higher dimensional case is still open.

(2) When $f\in W^{1,q}(\Omega)$ and $\theta=0$ or $1$, we give another new proof in Section \ref{inh-p} inspired by \cite{CC} for the complex setting. More precisely, we get a $C^2$-estimate of $u$ in terms of the $W^{2,p}$-bound (Theorem \ref{C2-est}), which makes \eqref{4-eq-inh} become a uniformly elliptic equation.  Then the classical theory of uniformly elliptic equations can be applied.
\end{rem}

In higher dimensions, the interior regularity and the Bernstein theorem are still widely open.
In fact, according to the counterexample in \cite{TW1} for the affine maximal surface equation, there may be no interior estimates if no further assumptions are made.  More precisely,  \eqref{detcond} may not hold. We give a partial result by assuming integral bounds on  the second derivatives and the inverse of the determinant.

\begin{theo}\label{main-thm}
Let $\Omega\subset\mathbb R^n$ be a convex domain and let
$u$ be a smooth uniformly convex  solution to equation \eqref{4-eq-h} with $\theta\in[0,1]$ on $\Omega$. Assume that $p, q>0$ satisfy $\frac{n}{p}+\frac{1}{q}<\frac{2}{n}$.
Then for any $\Omega'\Subset\Omega$,  there exists a constant $C$, depending only on $p$, $q$, $\theta$, $\sup_\Omega|u|$,
$\| u\|_{W^{2,p}\left(\Omega\right)}$, $\left\|(\det D^2u)^{-1}\right\|_{L^{q}\left(\Omega\right)}$ and $dist(\Omega',\partial\Omega)$, such that $$\|u\|_{C^{4,\alpha}(\Omega')}\leq C.$$
\end{theo}

As an application, we obtain a Liouville type theorem in higher dimensions.

\begin{cor}\label{Liouville}
Let $u$ be an entire smooth uniformly convex solution to \eqref{4-eq} with $\theta\in [0,1]$ on $\mathbb R^n$. Suppose there are $p, q, C>0$ such that $\frac{n}{p}+\frac{1}{q}<\frac{2}{n}$ and
\begin{equation}\label{uni-bound}
\int_{B_R(0)} |D^2u|^p+(\det D^2 u)^{-q} \,dx\leq CR^n\,,\quad \forall\,R>0.
\end{equation}
Then $u$ is a quadratic polynomial.
\end{cor}

%The H\"older regularities of \eqref{4-eq-inh} have been studied before in \cite{Lo,Le3}. In \cite{Lo}, under the assumption that $\det D^2u$ sufficiently close to a positive constant, the author used $W^{2,p}$-estimate of Monge-Amp\`ere equations, boundedness of degenerate elliptic equations and Caffarelli-Guti\'errez's Harnack inequality for homogeneous linearized Monge-Amp\`ere equations to obtain the H\"older estimate. In \cite{Le3}, the author considered \eqref{4-eq-inh} as  elliptic equations of divergence form, and used the integrability of Green's function and Moser's iteration to get the boundedness of $w$, then combing Caffarelli-Guti\'errez's Harnack inequality to obtain the H\"older estimate. In either case, they viewed \eqref{4-eq-inh} as second order equations to obtain the regularity. In this paper, we first view \eqref{4-eq-inh} as 4th order equations directly, and we will use the classical methods of interior estimate to get $W^{2,p}$-estimate of $u$, which is inspired by \cite{CC}, then we follow the same line in \cite{Lo,Le3} to get the H\"older estimate.

%Hence, by the classical Schauder theory we know
%\begin{cor}
%Under the assumption of Theorem \ref{main-thm} for any $0<r<1$, we have $\left\|D^{k} u\right\|_{L^\infty(B_{r})} \leq C(k)$, for any $k \geq 2 .$ Here $C(k)$ has the same dependence as described in Theorem \ref{main-thm} besides dependence on $r$ and $\|u\|_{L^\infty(B_{1})}$.
%\end{cor}

The structure of the paper is as follows. In Section \ref{neweq}, we apply the partial Legendre transform  to \eqref{4-eq} in dimension two to derive a new equation.  The key interior gradient estimate (Theorem \ref{grad}) for the new equation is established in Section \ref{grades}. Then we prove Theorem \ref{int-est} in Section \ref{maindim2} with this key estimate. In Section \ref{inh-p}, we first derive the $W^{2,p}$-estimate of $u$ (Theorem \ref{W2p-est}), and then we prove Theorem \ref{int-est-inh}. Section \ref{maindimh} is devoted to some study on interior regularity in higher dimensions.

{\bf Acknowledgments.} The authors would like to thank Guohuan Qiu for some valuable discussions.
\vskip 20pt

\section{The homogeneous equation in dimension two}

In this section, we present a new proof for the interior estimate for the homogeneous equation without Caffarelli-Guti\'errez's theory.

\subsection{The new equation under partial Legendre transform }\label{neweq}

We first focus on the dimension two case. Write $u(x)=u(x_1,x_2)$.
The partial Legendre transform in the $x_1$-variable is
\beq\label{p-leg}
u^\star(\xi, \eta)=x_1u_{x_1}(x_1, x_2)-u(x_1,x_2),
\eeq
where $$y=(\xi, \eta)=\mathcal P(x_1, x_2):=(u_{x_1}, x_2)\in \mathcal{P}(\Omega)=:\Omega^\star.$$
We have
$$
\frac{\partial(\xi,\eta)}{\partial(x_1,x_2)}=
\begin{pmatrix}
u_{x_1 x_1}  &\ \ u_{x_1 x_2}\\[4pt]
 0 &\ \ 1\\
\end{pmatrix}, \quad\text{and }
\
 \frac{\partial(x_1,x_2)}{\partial(\xi,\eta)}=
\begin{pmatrix}
\frac{1}{u_{x_1 x_1}}  &\ \  -\frac{u_{x_1 x_2}}{u_{x_1 x_1}}\\[4pt]
0 &\ \ 1\\
\end{pmatrix}.
$$
Hence,
\beq\label{rel-p}
u^\star_\xi=x_1,\ u^\star_{\xi\xi}=\frac{1}{u_{x_1 x_1}},\  u^\star_\eta=-u_{x_2},\ u^\star_{\eta\eta}=-\frac{\det D^2u}{u_{x_1 x_1}},\ u^\star_{\xi\eta}=-\frac{u_{x_1 x_2}}{u_{x_1 x_1}}.
\eeq
The partial Legendre transform has been used widely in the study of the Monge-Amp\`ere equation \cite{DS, F, GP, LS, Li}. Here we apply it to equation \eqref{4-eq-h}.

In order to derive the equation under the partial Legendre transform, we consider the associated functionals
of \eqref{4-eq-h}
$$A_\theta(u)=
\begin{cases}
\ \int_\Omega [\det D^2u]^{\theta}\,dx,& \theta> 0, \ \ \theta\neq 1,\\[4pt]
\ \int_\Omega\log\det D^2u\, dx, & \theta=0,\\[4pt]
\ \int_\Omega\det D^2u\log\det D^2u\, dx, & \theta=1.
\end{cases}$$
The case of $\theta=0$ is essentially included in \cite{LZ}.
\begin{prop}\label{new-eq}
Let $u$ be a uniformly convex solution to \eqref{4-eq-h} in $\Omega$. Then in $\Omega^\star=\mathcal P(\Omega)$, its partial Legendre transform $u^\star$ satisfies
\beq\label{eq-new1}
w^\star w^\star_{\xi\xi}+w^\star_{\eta\eta}+(\theta-1){w^\star_\xi}^2+\frac{\theta-2}{w^\star}{w^\star_\eta}^2=0,
\eeq
Here $w^\star=-\frac{u^\star_{\eta\eta}}{u^\star_{\xi\xi}}$.
\end{prop}

\begin{proof}
As
$$\det D^2u=-\frac{u^\star_{\eta\eta}}{u^\star_{\xi\xi}},\ \ dxdy=u^\star_{\xi\xi}\,d\xi d\eta,$$
we have
\begin{eqnarray*}
A_\theta(u)
&=&\int_{\Omega^\star}\left(-\frac{u^\star_{\eta\eta}}{u^\star_{\xi\xi}}\right)^{\theta} u^\star_{\xi\xi}\,d\xi d\eta\\
&=&\int_{\Omega^\star}(-u^\star_{\eta\eta})^{\theta} {u^\star_{\xi\xi}}^{1-\theta}\,d\xi d\eta=:A_\theta^\star(u^\star), \ \theta\in (0,1);\\[4pt]
A_0(u)
&=&\int_{\Omega^\star}\log\left(-\frac{u^\star_{\eta\eta}}{u^\star_{\xi\xi}}\right) u^\star_{\xi\xi}\,d\xi d\eta=:A_0^\star(u^\star);\\[4pt]
A_1(u)
&=&\int_{\Omega^\star}\left(-\frac{u^\star_{\eta\eta}}{u^\star_{\xi\xi}}\right)\log\left(-\frac{u^\star_{\eta\eta}}{u^\star_{\xi\xi}}\right) u^\star_{\xi\xi}\,d\xi d\eta=:A_1^\star(u^\star).
\end{eqnarray*}
Since $u$ is maximal with respect to the functional $A_{\theta}$, $u^\star$ is maximal with respect to the functional $A_{\theta}^\star$.
It suffices to derive the Euler-Lagrange equation of $A_\theta^\star$.  See \cite{TW1, Z1} for the case of the Legendre transform.

First, we consider $\theta\in(0,1)$.
For $\varphi\in C^\infty_0(\Omega^\star)$, by integration by parts,
\begin{eqnarray}
\left.\frac{d A_\theta^\star(u^\star+t\varphi)}{dt}\right|_{t=0}&=&\int_{\Omega^\star}[-\theta(-u^\star_{\eta\eta})^{\theta-1} {u^\star_{\xi\xi}}^{1-\theta}\varphi_{\eta\eta}+(1-\theta)(-u^\star_{\eta\eta})^{\theta} (u^\star_{\xi\xi})^{-\theta}\varphi_{\xi\xi}]\,d\xi d\eta\nonumber\\[3pt]
&=&\int_{\Omega^\star}\{-[\theta(-u^\star_{\eta\eta})^{\theta-1} {u^\star_{\xi\xi}}^{1-\theta}]_{\eta\eta}+[(1-\theta)(-u^\star_{\eta\eta})^{\theta} (u^\star_{\xi\xi})^{-\theta}]_{\xi\xi}\}\varphi\,d\xi d\eta.
\end{eqnarray}
Denote $w^\star=-\frac{u^\star_{\eta\eta}}{u^\star_{\xi\xi}}$.
Then the equation, after the transformation, becomes
$$-\theta ({w^\star}^{\theta-1})_{\eta\eta}+(1-\theta) ({w^\star}^{\theta})_{\xi\xi}=0.$$
After simplification, this is \eqref{eq-new1}. Similarly,
for  $\varphi\in C^\infty_0(\Omega^*)$,
\begin{eqnarray*}
\left.\frac{d A_0^\star(u^\star+t\varphi)}{dt}\right|_{t=0}&=&\int_{\Omega^\star}-\frac{u^\star_{\xi\xi}}{u^\star_{\eta\eta}} \left(-\frac{\varphi_{\eta\eta}u^\star_{\xi\xi}-u^\star_{\eta\eta}\varphi_{\xi\xi}}{{u^\star_{\xi\xi}}^2}\right)u^\star_{\xi\xi}+\log\left(-\frac{u^\star_{\eta\eta}}{u^\star_{\xi\xi}}\right) \varphi_{\xi\xi}\,d\xi d\eta\\[3pt]
&=&\int_{\Omega^\star} \frac{\varphi_{\eta\eta}u^\star_{\xi\xi}-u^\star_{\eta\eta}\varphi_{\xi\xi}}{u^\star_{\eta\eta}}+\log\left(-\frac{u^\star_{\eta\eta}}{u^\star_{\xi\xi}}\right) \varphi_{\xi\xi}\,d\xi d\eta,
\end{eqnarray*}
and the equation, after the transformation, becomes
$$-[(w^\star)^{-1}]_{\eta\eta}+ (\log w^\star)_{\xi\xi}=0.$$
After simplification, we obtain \eqref{eq-new1}.
Finally,
\begin{eqnarray*}
\left.\frac{d A_1^\star(u^\star+t\varphi)}{dt}\right|_{t=0}&=&\int_{\Omega^\star}(1+\log w^\star) \left(-\frac{\varphi_{\eta\eta}u^\star_{\xi\xi}-u^\star_{\eta\eta}\varphi_{\xi\xi}}{{u^\star_{\xi\xi}}^2}\right)u^\star_{\xi\xi}+w^\star\log w^\star \varphi_{\xi\xi}\,d\xi d\eta\\[3pt]
&=&-\int_{\Omega^\star} (1+\log w^\star)(w^\star\varphi_{\xi\xi}-\varphi_{\eta\eta})+w^\star\log w^\star \varphi_{\xi\xi}\,d\xi d\eta.
\end{eqnarray*}
Then the equation after transformation is
$$-w^\star_{\xi\xi}-(\log w^\star)_{\eta\eta}=0,$$
which is equivalent to \eqref{eq-new1}.
\end{proof}

\subsection{The interior gradient estimate of \eqref{eq-new1}}\label{grades}

For simplicity, we change notations in this section and write \eqref{eq-new1} as
\beq\label{eq-new3}
u u_{xx}+u_{yy}=(1-\theta)u_x^2+\frac{2-\theta}{u}u_y^2.
\eeq
It is easy to see that this is a quasi-linear equation with the right-hand side depending on the gradient.
We prove the following interior gradient estimate.
\begin{theo}\label{grad}
Assume $u$ is a solution to \eqref{eq-new3} with  $\theta\in [0,1]$ on $B_R:=B_R(0)$ and satisfies $0<\lambda\leq u\leq \Lambda$.
Then there exist $\alpha, C>0$ depending on $\lambda$, $\Lambda$, $R$ and $\theta$, such that
\beq\label{f0}
\int_{B_R}|\nabla u|^3(R^2-x^2-y^2)^{\alpha}\,dV\leq C.
\eeq
\end{theo}

\begin{proof}
Let $w=v\phi \eta$, where $$v=\sqrt{u_x^2+u_y^2+1},  \  \eta=(R^2-x^2-y^2)^\alpha,\ \alpha>3$$
 and $\phi=\phi(u)$ is a positive function of $u$ to be determined later.

A direct calculation yields
\begin{eqnarray}\label{f1}
uw_{xx}+w_{yy}&=& (uv_{xx}+v_{yy})\phi\eta+2[uv_x(\phi\eta)_x+v_y(\phi\eta)_y]+ [u(\phi\eta)_{xx}+(\phi\eta)_{yy}]v.
\end{eqnarray}
It is clear that
\begin{eqnarray}
v_x &=& v^{-1}(u_x u_{xx}+u_y u_{xy}),\quad
v_y = v^{-1}(u_x u_{xy}+u_y u_{yy}),\label{f2}\\
v_{xx} &=& -v^{-3}(u_x u_{xx}+u_y u_{xy})^2+v^{-1}(u_{xx}^2+u_x u_{xxx}+u_{xy}^2+u_y u_{xxy}),\label{f2-2}\\
v_{yy} &=& -v^{-3}(u_x u_{xy}+u_y u_{yy})^2+v^{-1}(u_{xy}^2+u_x u_{xyy}+u_{yy}^2+u_y u_{yyy}).\label{f2-3}
\end{eqnarray}
Differentiating the equation, we have
\begin{eqnarray}
u_xu_{xx}+uu_{xxx}+u_{xyy}&=&2(1-\theta)u_xu_{xx}-\frac{2-\theta}{u^2}u_xu_y^2+\frac{2(2-\theta)}{u}u_yu_{xy},\label{f3}\\
u_yu_{xx}+uu_{xxy}+u_{yyy}&=&2(1-\theta)u_xu_{xy}-\frac{2-\theta}{u^2}u_yu_y^2+\frac{2(2-\theta)}{u}u_yu_{yy}.\label{f4}
\end{eqnarray}
By the Cauchy inequality, we have
\beq\label{f5}
(u_x u_{xx}+u_y u_{xy})^2\leq (v^2-1)(u_{xx}^2+u_{xy}^2),\ \
(u_x u_{xy}+u_y u_{yy})^2\leq (v^2-1)(u_{xy}^2+u_{yy}^2).
\eeq
Then by \eqref{f2}, \eqref{f2-2}, \eqref{f2-3}, \eqref{f3}, \eqref{f4} and \eqref{f5},
the first term in \eqref{f1} satisfies
\begin{eqnarray*}
&&uv_{xx}+v_{yy}\\[5pt]
&=&u[v^{-1}(u^2_{xx}+u^2_{xy})-v^{-3}(u_x u_{xx}+u_y u_{xy})^2]+v^{-1}u(u_xu_{xxx}+u_yu_{xxy})\\[5pt]
&&+[v^{-1}(u^2_{xy}+u^2_{yy})-v^{-3}(u_x u_{xy}+u_y u_{yy})^2]+v^{-1}(u_xu_{xyy}+u_yu_{yyy}))\\[5pt]
&\geq& v^{-1}\left[2(1-\theta)u_{x}(u_xu_{xx}+u_yu_{xy})+\frac{2(2-\theta)}{u}u_y(u_x u_{xy}+u_y u_{yy})-(v^2-1)u_{xx}\right]\\[5pt]
&&\ -v^{-1}\frac{2-\theta}{u^2}(u_x^2u_y^2+u_y^4)+v^{-3}\left[u(u^2_{xx}+u^2_{xy})+u^2_{xy}+u^2_{yy}\right]\\[5pt]
&=&2(1-\theta)u_x v_x+\frac{2(2-\theta)}{u}u_y v_y-\frac{2-\theta}{u^2}\frac{v^2-1}{v} u_y^2-\frac{v^2-1}{v} u_{xx}\\[5pt]
&&\ +v^{-3}\left[u(u^2_{xx}+u^2_{xy})+u^2_{xy}+u^2_{yy}\right]\\[5pt]
&\geq &2(1-\theta)u_x v_x+\frac{2(2-\theta)}{u}u_y v_y-\frac{2-\theta}{u^2}\frac{v^2-1}{v} u_y^2-v u_{xx}+v^{-1}u_{xx}+v^{-3}uu_{xx}^2\\[5pt]
&\geq &2(1-\theta)u_x v_x+\frac{2(2-\theta)}{u}u_y v_y-\frac{2-\theta}{u^2}\frac{v^2-1}{v} u_y^2-v u_{xx}-\frac{v}{4u}-v^{-3}uu_{xx}^2+v^{-3}uu_{xx}^2\\[5pt]
&\geq &2(1-\theta)u_x v_x+\frac{2(2-\theta)}{u}u_y v_y-\frac{2-\theta}{u^2}\frac{v^2-1}{v} u_y^2-v u_{xx}-Cv,
\end{eqnarray*}
where $C>0$ depends on $\lambda$. By integration by parts,
\begin{eqnarray*}
&&\int_{B_R} (uv_{xx}+v_{yy})\phi\eta\,dV\\[5pt]
&\geq&
-\int_{B_R}2(1-\theta)[u_{xx} v\phi\eta+u_xv(\phi\eta)_{x}]\,dV\\[5pt]
&&-\int_{B_R}2(2-\theta)\left[\frac{u_{yy}}{u}v \phi\eta-\frac{u_{y}^2}{u^2}v \phi\eta+\frac{1}{u}u_yv(\phi\eta)_{y}\right]\,dV\\[5pt]
&&-\int_{B_R}\left(\frac{2-\theta}{u^2}\frac{v^2-1}{v} u_y^2\phi\eta+v u_{xx}\phi\eta+C\phi\eta v\right) \,dV\\[5pt]
&=&\int_{B_R}u_{xx} v\phi\eta\,dV-\int_{B_R}2(2-\theta)\frac{1}{u}\left[(1-\theta)u_x^2+\frac{2-\theta}{u}u_y^2\right]v\phi\eta\,dV\\[5pt]
&&-\int_{B_R}\left[2(1-\theta)u_xv(\phi\eta)_{x}+2(2-\theta)\left(\frac{u_yv(\phi\eta)_{y}}{u}-\frac{u_{y}^2v \phi\eta}{u^2}\right)\right]\,dV\\[5pt]&&-\int_{B_R}\left[\frac{2-\theta}{u^2}\left(v-\frac{1}{v}\right) u_y^2\phi\eta+C\phi\eta v\right]\,dV\\[5pt]
&=&\int_{B_R}u_{xx} v\phi\eta\,dV-\int_{B_R}\left[2(1-\theta)\phi'+\frac{2(2-\theta)(1-\theta)}{u}\phi\right]u_x^2v\eta\,dV\\[5pt]
&&-\int_{B_R}\left[2(2-\theta)\frac{\phi'}{u}+(2-\theta)(3-2\theta)\frac{1}{u^2}\phi\right]u_y^2v\eta\,dV\\[5pt]
&&-\int_{B_R}\left[2(1-\theta)u_xv\phi\eta_{x}+\frac{2(2-\theta)}{u}u_y v\phi\eta_{y}+C\phi\eta v\right]\,dV+\int_{B_R}\frac{2-\theta}{u^2}v^{-1} u_y^2\phi\eta\,dV\\[5pt]
&\geq&\int_{B_R}u_{xx} v\phi\eta\,dV-\int_{B_R}\left[2(1-\theta)\phi'+\frac{2(2-\theta)(1-\theta)}{u}\phi\right]u_x^2v\eta\,dV\\[5pt]
&&-\int_{B_R}\left[2(2-\theta)\frac{\phi'}{u}+(2-\theta)(3-2\theta)\frac{1}{u^2}\phi\right]u_y^2v\eta\,dV\\[5pt]
&&-C\left[\int_{B_R}|\nabla u|^2(R^2-x^2-y^2)^{\alpha-1}\,dV+\int_{B_R}|\nabla u|(R^2-x^2-y^2)^{\alpha}\,dV+1\right].
\end{eqnarray*}
Here $C>0$ depends on $\lambda$, $\Lambda$, $R$ and $\theta$.
For the third term in \eqref{f1}
\beqs
&&[u(\phi\eta)_{xx}+(\phi\eta)_{yy}]v\\[5pt]
&=&v\left[\phi''(u)(uu^2_x+u_y^2)\eta+\phi'(u)(uu_{xx}+u_{yy})\eta+2\phi'(u)(uu_x\eta_x+u_y\eta_y)
+\phi\cdot(u\eta_{xx}+\eta_{yy})\right]\\[5pt]
&=&v\bigg[\phi''(u)(uu^2_x+u_y^2)\eta+\phi'(u)\left((1-\theta)u_x^2+\frac{2-\theta}{u}u_y^2\right)\eta+2\phi'(u)(uu_x\eta_x+u_y\eta_y)
\\[5pt]&&\ +\phi\cdot(u\eta_{xx}+\eta_{yy})\bigg]\\[5pt]
&\leq& \left[(\phi''u+(1-\theta)\phi') u^2_x+\left(\phi''+\phi' \frac{2-\theta}{u}\right)u_y^2\right]v\eta\\[5pt]
&&\ +C\left[|\nabla u|^2(R^2-x^2-y^2)^{\alpha-1}+|\nabla u|(R^2-x^2-y^2)^{\alpha-2}\right].
\eeqs
Here $C>0$ depends on $\lambda$, $\Lambda$, $R$ and $\theta$.
Integrating by parts, the second term in \eqref{f1}
$$2\int_{B_R} uv_x(\phi\eta)_x+v_y(\phi\eta)_y\,dV
=-2\int_{B_R} u_x(\phi\eta)_xv\,dV-2\int_{B_R} [u(\phi\eta)_{xx}+(\phi\eta)_{yy}]v\,dV,$$
and hence
\begin{eqnarray}\label{f7}
&&\int_{B_R} uw_{xx}+w_{yy}\,dV\nonumber\\[5pt]
&=&\int_{B_R} (uv_{xx}+v_{yy})\phi\eta\,dV-2\int_{B_R} u_x(\phi\eta)_xv\,dV-\int_{B_R} [u(\phi\eta)_{xx}+(\phi\eta)_{yy}]v\,dV\nonumber\\[5pt]
&\geq&\int_{B_R}u_{xx} v\phi\eta\,dV+\int_{B_R}[(-5+3\theta)\phi'-\frac{2(2-\theta)(1-\theta)}{u}\phi-\phi''u]u_x^2v\eta\,dV\nonumber\\[5pt]
&&+\int_{B_R}[-3(2-\theta)\frac{\phi'}{u}-(2-\theta)(3-2\theta)\frac{\phi}{u^2}-\phi'']u_y^2v\eta\,dV\nonumber\\[5pt]
&&-C\left[\int_{B_R}|\nabla u|^2(R^2-x^2-y^2)^{\alpha-1}\,dV+\int_{B_R}|\nabla u|(R^2-x^2-y^2)^{\alpha-2}\,dV+1\right].
\end{eqnarray}
Note that the left hand side term satisfies
\beq\label{f8}
\int_{B_R} uw_{xx}+w_{yy}\,dV=\int_{B_R} u_{xx}v\phi\eta\,dV.
\eeq
Now we choose
$$\phi(u)=Au^{\theta-2}-\frac{u}{2\theta^2-9\theta+9}$$
with
$$A\geq\frac{\Lambda^{3-\theta}}{2\theta^2-9\theta+9}+1.$$
Then it is clear that $\phi(u)>0$. Furthermore, since $\theta\in[0,1]$, we have
$$(-5+3\theta)\phi'-\frac{2(2-\theta)(1-\theta)}{u}\phi-\phi''u=1$$
and
$$-3(2-\theta)\frac{\phi'}{u}-(2-\theta)(3-2\theta)\frac{\phi}{u^2}-\phi''=\frac{2(2-\theta)(3-\theta)}{(2\theta^2-9\theta+9)u}\geq C_0>0.$$
Combining them with \eqref{f7} and \eqref{f8},
we obtain
\begin{eqnarray*}
&&\int_{B_R}|\nabla u|^3(R^2-x^2-y^2)^{\alpha}\,dV\\[5pt]
&\leq&  C_1\int_{B_R}|\nabla u|^2(R^2-x^2-y^2)^{\alpha-1}\,dV+C_2\int_{B_R}|\nabla u|(R^2-x^2-y^2)^{\alpha-2}\,dV+C_3\\[5pt]
&\leq & C_1'\left(\int_{B_R}|\nabla u|^3(R^2-x^2-y^2)^{\alpha}\,dV\right)^{\frac{2}{3}}+
C_2'\left(\int_{B_R}|\nabla u|^3(R^2-x^2-y^2)^{\alpha}\,dV\right)^{\frac{1}{3}}+C_3.
\end{eqnarray*}
Hence, \eqref{f0} follows.
\end{proof}

\subsection{Proof of Theorem \ref{int-est}}\label{maindim2}

In order to use the partial Legendre transform, we first recall the modulus of convexity.
For a convex function on $\mathbb R^n$, the {\it modulus of convexity} $m_u$ of $u$ is defined by
\beq\label{mdc}
m_u(t)=\inf\{u(x)-\ell_z(x):|x-z|>t\},
\eeq
where $t>0$ and $\ell_z$ is the supporting function of $u$ at $z$. For a strictly convex function, $m_u$ must be a positive function. A result of Heinz \cite{H} implies that in two dimensions, when $\det D^2u\geq \lambda>0$, there exists a positive function $C(t)>0$ depending on $\lambda$ such that $m_u(t)\geq C(t)>0$. Now for the partial Legendre transform \eqref{p-leg}, we consider the mapping
\beq
(\xi, \eta)=\mathcal P(x, y)=(u_x,y): B_R(0)\to \mathbb R^2.
\eeq
The following important property is revealed in \cite{Li}.
\begin{lem}[{\cite[Lemma 2.1]{Li}}]\label{mod}
There exists a constant $\delta>0$ depending on the modulus of convexity $m_u$ defined in \eqref{mdc}, such that $B_\delta(0)\subset \mathcal P(B_R(0))$.
\end{lem}

\begin{proof}[Proof of Theorem \ref{int-est}]
For any $p\in \Omega$, let $R=\frac{dist(p,\partial\Omega)}{2}$. Without loss of generality, we assume $\mathcal P(p)=0$.  It is clear that $\sup_{B_R(p)}|Du|\leq C$ for some constant $C$ depending on $R$ and $\sup_\Omega |u|$.
By Lemma \ref{mod}, there exists $\delta>0$ such that $B_\delta(0)\subset \mathcal P(B_R(p))$.
According to Proposition \ref{new-eq}, $u^\star$ satisfies \eqref{eq-new1} in $B_\delta(0)$ with $$0<\lambda\leq w^\star=-\frac{u^\star_{\eta\eta}}{u^\star_{\xi\xi}}\leq \Lambda.$$
By Theorem \ref{grad}, $$\| w^\star\|_{W^{1, 3}(B_{\frac{7\delta}{8}}(0))}\leq C.$$
Note that $n=2$.
By the Sobolev theorem, we have the $C^\alpha$ estimate of $w^\star$. And by the interior $W^{2,p}$-estimate of the uniformly elliptic equation \eqref{eq-new1}, we have 
the estimate $$\| w^\star\|_{W^{2, \frac{3}{2}}(B_{\frac{3\delta}{4}}(0))}\leq C,$$ which implies
the $W^{1, 6}$-estimate of $w^\star$. Again by the interior $W^{2,p}$-estimate of the uniformly elliptic equation \eqref{eq-new1}, we have
the $$\| w^\star\|_{W^{2, 3}(B_{\frac{\delta}{2}}(0))}\leq C,$$ which implies the
 $C^{1,\alpha}$ estimate of $w^\star$. Then by the Schauder estimate of \eqref{eq-new1}, we have
 $$\|w^\star\|_{C^{2,\alpha}(B_{\frac{\delta}{4}}(0))}\leq C$$ and all the higher order estimates of $u^\star$.
Transforming back by the partial Legendre transform, we obtain the lower bound of $u_{x_1x_1}$ by \eqref{rel-p}. Since we can do partial Legendre transforms of $u$ in any direction, we can obtain the lower bound  for the smallest eigenvalue of $D^2u$, which implies the boundedness of $D^2u$ by \eqref{detcond}. Then we have all the higher order estimates of $u$.
\end{proof}

\vskip 20pt

\section{The inhomogeneous equations in dimension two}\label{inh-p}
In this section, we will study the interior estimate for the inhomogeneous equation \eqref{4-eq-inh}.

We first recall the regularity theory of the second order elliptic equation in divergence form
\beq\label{de-eq}
D_j(a_{ij}(x)D_i u)=D_ig^i + h\quad\text{in}\,\,\Omega,
\eeq
where $g=(g^1, \cdots, g^n)$ is a vector valued function and $\{a_{ij}(x)\}$ satisfies
$$0< \lambda(x)|\xi|^2\leq a_{ij}(x)\xi_i\xi_j\leq \Lambda(x)|\xi|^2,\ \ \forall\ \xi\in\mathbb R^n\setminus\{0\}$$
and can be discontinuous. When $\{a_{ij}(x)\}$ is uniformly elliptic,  a fundamental regularity theory
was established by  De Giorgi, Nash, Moser, etc.
In \cite{MS,Tr}, the classical De Giorgi-Nash-Moser's theory was extended to degenerate linear elliptic equation with $\lambda(x)^{-1}\in L^p(\Omega)$ and $\Lambda(x)\in L^q(\Omega)$ where $\frac{1}{p}+\frac{1}{q}<\frac{2}{n}$. We denote by $S_n^+$ the set of $n\times n$ nonnegative definite matrices.

\begin{theo}[{\cite[Theorem 4.2]{Tr}}]\label{thm4.2}
Let $\{a_{ij}\}:\Omega\to S_n^+$ be such that $\lambda^{-1}\in L_{loc}^p(\Omega)$ for some $p>n$. Let $g\in L^{\infty}(\Omega;\mathbb{R}^n)$ and $h\in L^q(\Omega)$. Assume $p$ and $q$ satisfy $\frac{1}{q}+\frac{1}{p}<\frac{2}{n}$. Suppose that $u$ is a subsolution (supersolution)  to \eqref{de-eq} in $B_R=B_R(y)\subset\Omega$  and that $u\leq 0(u\geq 0)$ on $\partial B_R$. Then
$$\sup_{B_R}u(-u)\leq C\|\lambda^{-1}\|_{L^p(B_R)}\left(\|g\|_{L^\infty(B_R)}R^{1-\frac{n}{p}}+\|h\|_{L^q(B_R)}R^{2-\frac{n}{q}-\frac{n}{p}}\right).$$
\end{theo}

In \cite{Lo}, the author investigated the linearized Monge-Amp\`ere equation with right hand side $f=D_ig^i+h$,
i.e. $a_{ij}(x)=U^{ij}$ in \eqref{de-eq}.
Under the stronger assumption that $\det D^2u$ is sufficiently close to a positive constant,  he can use the $W^{2,p}$-estimate of the Monge-Amp\`ere equation and Theorem \ref{thm4.2} to obtain the interior regularity.
It was later shown in \cite{Le3} that when $n=2$ we only need the determinant to be bounded from above and below.

Now we turn to the fourth order equation
\eqref{4-eq-inh}.  By the divergence free property of $U^{ij}$, i.e.,  $\sum_iD_i U^{ij}=0$,  we can rewrite equation \eqref{4-eq-inh} in divergence form
\beq\label{div2}
D_j(U^{ij}D_iw)=D_ig^i+h.
\eeq
Note that in dimension two,  it holds $\frac{\det D^2u}{\Delta u}I\leq U^{ij}\leq \Delta u I$.
 In view of Theorem \ref{thm4.2}, it suffices to get the $L^p$-bound of $\Delta u$ for sufficiently large $p$. %\subsection{$W^{2,p}$-estimate of $u$}\label{maindim2-H}
In the following, we will use an integral method directly in \eqref{4-eq-inh} to derive it.
Instead of the classical Sobolev inequality, we will need the following Monge-Amp\`ere Sobolev inequality.

\begin{lem}[{\cite[Proposition 2.6]{Le3}}]\label{Sobolev}
Assume $n=2$.  Let $u$ be a smooth, strictly convex function defined in a neighborhood of a bounded domain $\Omega\subset\mathbb R^2$. Suppose $u$ satisfies
\begin{equation}\label{detcond-1}
0<\lambda\leq\det D^2u\leq\Lambda.
\end{equation}
Then for any $\chi> 2$ there exists a constant $C>0$, depending only on $\lambda$, $\Lambda$ and $\chi$, such that
$$\left(\int_{\Omega}|v|^\chi\,dx\right)^{\frac{1}{\chi}}\leq C\left(\int_{\Omega}U^{ij}v_iv_j\,dx\right)^{\frac{1}{2}}\quad\text{for\,\,all}\,\,v\in C_0^\infty(\Omega).$$
\end{lem}

The above inequality is the two dimensional counterpart of the Monge-Amp\`ere Sobolev inequality in higher dimensions derived by Tian and Wang \cite{TW}.

\begin{theo}\label{W2p-est}
Assume $n=2$ and $\theta\in [0,1]$. Let $g\in L^{\infty}(\Omega;\mathbb{R}^n)$ and $h\in L^q(\Omega)$. Let $u$ be a uniformly convex smooth solution to equation \eqref{div2} in $B_R$ satisfying
\eqref{detcond-1}.
Then for any $p\geq 1$, there exists a constant $C>0$ depending on $\lambda$, $\Lambda$, $R$, $\theta$, $\|g\|_{L^\infty(B_R)}$, $\|h\|_{L^q(B_R)}$ and $p$ such that
$$\|\Delta u\|_{L^p(B_{R/2})}\leq C.$$
\end{theo}

\begin{proof}
We first consider the case $\theta\in [0,1)$. Denote
$$v=\Delta u,\,w=(\det D^2u)^{-(1-\theta)},\, \rho=R^2-|x|^2.$$
We consider $z=v^\beta\varphi\rho^\alpha$,
where  $\alpha, \beta>0$ are constants and $\varphi=\varphi(w)$ is a positive function to be determined.
Then
\begin{eqnarray*}
z_i&=&\beta v_iv^{\beta-1}\varphi\rho^\alpha+v^{\beta}(\varphi\rho^\alpha)_i,\\
z_{ij}&=&\beta(\beta-1)v_iv_jv^{\beta-2}\varphi\rho^\alpha+\beta v_{ij}v^{\beta-1}\varphi\rho^\alpha
+\beta v_iv^{\beta-1}(\varphi\rho^\alpha)_j+\beta v_jv^{\beta-1}(\varphi\rho^\alpha)_i+v^\beta(\varphi\rho^\alpha)_{ij}.
\end{eqnarray*}
By integration by parts with \eqref{4-eq-inh} and choosing $\alpha>2$, we have
\begin{eqnarray}
0&=&\int_{B_R} U^{ij}z_{ij}\,dx\nonumber\\[5pt]
&=&\beta(\beta-1)\int_{B_R} U^{ij}v_iv_j \varphi\rho^\alpha v^{\beta-2}\,dx+\beta\int_{B_R} U^{ij}v_{ij}\varphi \rho^{\alpha}v^{\beta-1}\,dx+ 2\beta\int_{B_R} U^{ij}v_iw_j\varphi' \rho^\alpha v^{\beta-1}\,dx
\nonumber\\[5pt]
&&+2\alpha \beta\int_{B_R} U^{ij} v_i\rho_j\varphi \rho^{\alpha-1}v^{\beta-1}\,dx+\int_{B_R} U^{ij}w_iw_j\varphi''\rho^{\alpha}v^\beta\,dx
+2\alpha\int_{B_R} U^{ij}w_i\rho_j\varphi' \rho^{\alpha-1}v^\beta\,dx\label{fm1}
\\[5pt]
&&+\int_{B_R} f\varphi'\rho^\alpha v^\beta\,dx+\alpha\int_{B_R} U^{ij}\rho_{ij}\varphi\rho^{\alpha-1}v^\beta\,dx+\alpha(\alpha-1)\int_{B_R} U^{ij}\rho_i\rho_j\varphi\rho^{\alpha-2}v^\beta\,dx.\nonumber
\end{eqnarray}
Note that for any $m$,
\beqs
u^{ij}(u_{mm})_{ij}=u^{il}u^{kj}u_{mij}u_{mkl}+(\ln\det D^2u)_{mm}
=u^{il}u^{kj}u_{mij}u_{mkl}+\frac{1}{1-\theta}\frac{w_m^2}{w^2}-\frac{1}{1-\theta}\frac{w_{mm}}{w}.
\eeqs
We have
\begin{eqnarray*}
\int_{B_R} U^{ij}v_{ij}\varphi \rho^{\alpha}v^{\beta-1}&&\,dx=\int_{B_R} \det D^2u u^{il}u^{kj}u_{mij}u_{mkl}\varphi \rho^{\alpha}v^{\beta-1}\,dx\\[5pt]
&&+\frac{1}{1-\theta}\int_{B_R} \det D^2u\frac{|\nabla w|^2}{w^2}\varphi \rho^{\alpha}v^{\beta-1}\,dx-\frac{1}{1-\theta}\int_{B_R} \det D^2u \frac{\triangle w}{w}\varphi \rho^{\alpha}v^{\beta-1}\,dx.
\end{eqnarray*}
By integration by parts, we have
\begin{eqnarray}
&&\int_{B_R} \det D^2u \frac{\triangle w}{w}\varphi \rho^{\alpha}v^{\beta-1}\,dx=\int_{B_R} w^{-\frac{2-\theta}{1-\theta}}\triangle w\varphi \rho^\alpha v^{\beta-1}\,dx\nonumber\\[5pt]
&=&\frac{2-\theta}{1-\theta}\int_{B_R} w^{-\frac{3-2\theta}{1-\theta}}|\nabla w|^2\varphi\rho^{\alpha}v^{\beta-1}\,dx-\int_{B_R} w^{-\frac{2-\theta}{1-\theta}}|\nabla w|^2\varphi' \rho^\alpha v^{\beta-1}\,dx \label{fm5}\\[5pt]
&&-\alpha\int_{B_R} w^{-\frac{2-\theta}{1-\theta}}w_i\rho_i\varphi\rho^{\alpha-1}v^{\beta-1}\,dx-(\beta-1)\int_{B_R} w^{-\frac{2-\theta}{1-\theta}}w_iv_i\varphi\rho^{\alpha}v^{\beta-2}\,dx.\nonumber
\end{eqnarray}
By changing coordinates at each point, say $x_0$, we can assume that $D^2u(x_0)$ is diagonal. Then
\begin{eqnarray}
u^{il}u^{kj}u_{mij}u_{mkl}=\frac{u_{mij}^2}{u_{ii}u_{jj}}\geq\frac{u_{mmi}^2}{u_{ii}u_{mm}}\geq \frac{v_i^2}{u_{ii}v}=\frac{u^{ij}v_iv_j}{v},\label{fm6}
\end{eqnarray}
where we used the Cauchy inequality in the second inequality.
Hence by \eqref{fm5} and \eqref{fm6}, we have
\begin{eqnarray}
\int_{B_R} U^{ij}v_{ij}\varphi \rho^{\alpha}v^{\beta-1}&&\,dx\geq \int_{B_R} U^{ij}v_iv_j\varphi\rho^\alpha v^{\beta-2}\,dx+\frac{1}{1-\theta}\int_{B_R} w^{-\frac{2-\theta}{1-\theta}}(\varphi'-\frac{\varphi w^{-1}}{1-\theta})|\nabla w|^2 \rho^\alpha v^{\beta-1}\,dx\nonumber\\[5pt]
&&+ \frac{\alpha}{1-\theta}\int_{B_R} w^{-\frac{2-\theta}{1-\theta}}w_i\rho_i\varphi\rho^{\alpha-1}v^{\beta-1}\,dx+\frac{\beta-1}{1-\theta}\int_{B_R} w^{-\frac{2-\theta}{1-\theta}}w_iv_i\varphi\rho^{\alpha}v^{\beta-2}\,dx.\label{fm2}
\end{eqnarray}
Then putting \eqref{fm2} into \eqref{fm1} yields
\begin{eqnarray}
0&\geq&\beta^2\int_{B_R} U^{ij}v_iv_j \varphi\rho^\alpha v^{\beta-2}\,dx+\frac{\beta}{1-\theta}\int_{B_R} w^{-\frac{2-\theta}{1-\theta}}(\varphi'-\frac{\varphi w^{-1}}{1-\theta})|\nabla w|^2 \rho^\alpha v^{\beta-1}\,dx\nonumber\\[5pt]
&&+ \frac{\alpha \beta}{1-\theta}\int_{B_R} w^{-\frac{2-\theta}{1-\theta}}w_i\rho_i\varphi\rho^{\alpha-1}v^{\beta-1}\,dx+\frac{\beta(\beta-1)}{1-\theta}\int_{B_R} w^{-\frac{2-\theta}{1-\theta}}w_iv_i\varphi\rho^{\alpha}v^{\beta-2}\,dx\nonumber\\[5pt]
&&+ 2\beta\int_{B_R} U^{ij}v_iw_j\varphi' \rho^\alpha v^{\beta-1}\,dx+2\alpha \beta\int_{B_R} U^{ij} v_i\rho_j\varphi \rho^{\alpha-1}v^{\beta-1}\,dx\label{fm3} \\[5pt]
&&+\int_{B_R} U^{ij}w_iw_j\varphi''\rho^{\alpha}v^\beta\,dx+\int_{B_R} f\varphi'\rho^\alpha v^\beta\,dx+2\alpha\int_{B_R} U^{ij}w_i\rho_j\varphi' \rho^{\alpha-1}v^\beta\,dx\nonumber\\[5pt]
&&+\alpha\int_{B_R} U^{ij}\rho_{ij}\varphi\rho^{\alpha-1}v^\beta\,dx+\alpha(\alpha-1)\int_{B_R} U^{ij}\rho_i\rho_j\varphi\rho^{\alpha-2}v^\beta\,dx.\nonumber
\end{eqnarray}
Now we choose $\varphi(w)=e^{e^{Aw}}$, where $A>0$ is to be determined later. Then we know that
\beq\label{phi'}
\varphi'(w)=Ae^{Aw}\varphi,\
\varphi''(w)=(A^2e^{Aw}+A^2e^{2Aw})\varphi.
\eeq
We will estimate the right hand side of \eqref{fm3} term by term.

By  \eqref{phi'}, the two terms of the third line in \eqref{fm3} satisfy
\begin{eqnarray*}
&&\left|2\beta\int_{B_R} U^{ij}v_iw_j\varphi' \rho^\alpha v^{\beta-1}\,dx\right|=\left|\int_{B_R} 2\beta Ae^{Aw} U^{ij}v_iw_j \varphi\rho^\alpha v^{\beta-1}\,dx\right|\\[5pt]
&& \ \ \ \ \leq \int_{B_R} \beta^2\frac{e^{Aw}}{1/2+e^{Aw}} U^{ij}v_iv_j \varphi\rho^\alpha v^{\beta-2}\,dx+\int_{B_R}(\frac{A^2}{2}e^{Aw}+A^2e^{2Aw}) U^{ij}w_iw_j \varphi \rho^\alpha v^{\beta}\,dx,
\end{eqnarray*}
and
\begin{eqnarray*}
\left|2\alpha \beta\int_{B_R} U^{ij} v_i\rho_j\varphi \rho^{\alpha-1}v^{\beta-1}\,dx\right|
&\leq&\int_{B_R} \frac{\beta^2}{8(1+2e^{Aw})} U^{ij}v_iv_j\varphi\rho^\alpha v^{\beta-2}\,dx\\[5pt]&&+\int_{B_R} 8(1+2e^{Aw})\alpha^2 U^{ij}\rho_i\rho_j\varphi\rho^{\alpha-2}v^\beta\,dx.
\end{eqnarray*}
Note that
$$|\nabla\rho|^2\leq \frac{\triangle u}{\det D^2u}U^{ij}\rho_i\rho_j, \ \ |\nabla v|^2\leq \frac{\triangle u}{\det D^2u}U^{ij}v_iv_j.$$
The two terms of  the second line in \eqref{fm3} satisfy
\begin{eqnarray*}
&&\left|\frac{\alpha \beta}{1-\theta}\int_{B_R} w^{-\frac{2-\theta}{1-\theta}}w_i\rho_i\varphi\rho^{\alpha-1}v^{\beta-1}\,dx\right|\\[5pt]
&\leq& \int_{B_R} \frac{\beta}{(1-\theta)^2} w^{-\frac{3-2\theta}{1-\theta}}|\nabla w|^2 \varphi\rho^\alpha v^{\beta-1}\,dx+\int_{B_R} \frac{1}{4}\alpha^2\beta w^{-\frac{1}{1-\theta}} |\nabla\rho|^2\varphi \rho^{\alpha-2}v^{\beta-1}\,dx\\[5pt]
&\leq&\int_{B_R} \frac{\beta}{(1-\theta)^2} w^{-\frac{3-2\theta}{1-\theta}}|\nabla w|^2 \varphi\rho^\alpha v^{\beta-1}\,dx+\int_{B_R} \frac{1}{4}\alpha^2\beta
U^{ij}\rho_i\rho_j \varphi \rho^{\alpha-2}v^{\beta}\,dx,
\end{eqnarray*}
and
\begin{eqnarray*}
&&\left|\frac{\beta(\beta-1)}{1-\theta}\int_{B_R} w^{-\frac{2-\theta}{1-\theta}}w_iv_i\varphi\rho^{\alpha}v^{\beta-2}\,dx\right|\\[5pt]
&\leq&\int_{B_R} \frac{(\beta-1)^2}{3(1-\theta)^2}(1+2e^{Aw})w^{-\frac{3-2\theta}{1-\theta}}|\nabla w|^2 \varphi\rho^\alpha v^{\beta-1}\,dx+\int_{B_R} \frac{3\beta^2w^{-\frac{1}{1-\theta}}}{4(1+2e^{Aw})}|\nabla v|^2 \varphi\rho^\alpha v^{\beta-3}\,dx\\[5pt]
&\leq&\int_{B_R}  \frac{(\beta-1)^2}{3(1-\theta)^2}(1+2e^{Aw})w^{-\frac{3-2\theta}{1-\theta}}|\nabla w|^2 \varphi\rho^\alpha v^{\beta-1}\,dx+\int_{B_R} \frac{3\beta^2}{4(1+2e^{Aw})} U^{ij}v_iv_j \varphi\rho^{\alpha}v^{\beta-2}\,dx.
\end{eqnarray*}
 By \eqref{phi'}, the last term of the fourth line in \eqref{fm3} satisfies
\begin{eqnarray*}
\left|2\alpha \int_{B_R} U^{ij} w_i\rho_j\varphi' \rho^{\alpha-1}v^{\beta}\,dx\right|
\leq \int_{B_R} \frac{A^2}{2}e^{Aw} U^{ij}w_iw_j \varphi\rho^{\alpha}v^\beta\,dx&+&\int_{B_R} 2\alpha^2 e^{Aw} U^{ij}\rho_i\rho_j \varphi\rho^{\alpha-2} v^\beta\,dx.
\end{eqnarray*}
Note that $f=\div(g)+h$, where $g$ is a bounded vector field. We denote $|g|=\sqrt{\sum_i (g^i)^2}$. Then by integration by parts, we have
\begin{eqnarray*}
&&\left|\int_{B_R}f\varphi'\rho^\alpha v^\beta\,dx\right|\\[5pt]
&=&\left|\int_{B_R} h\varphi'\rho^\alpha v^\beta\,dx-\int_{B_R} g^iw_i\varphi''\rho^\alpha v^\beta\,dx-\alpha\int_{B_R} g^i\rho_i \varphi'\rho^{\alpha-1}v^\beta\,dx-\beta\int_{B_R} g^iv_i\varphi'\rho^\alpha v^{\beta-1}\,dx\right|\\[5pt]
&\leq&\int_{B_R} \frac{1}{2}w^{-\frac{3-2\theta}{1-\theta}}|\nabla w|^2\varphi \rho^\alpha v^{\beta-1}\,dx+\int_{B_R} \frac{1}{2}(A^2e^{Aw}+A^2e^{2Aw})^2w^{\frac{3-2\theta}{1-\theta}}|g|^2\varphi\rho^\alpha v^{\beta+1}\,dx\\[5pt]
&&+\int_{B_R} \frac{\alpha}{2}w^{-\frac{1}{1-\theta}}|\nabla \rho|^2Ae^{Aw}\varphi\rho^{\alpha-2}v^{\beta-1}\,dx
+\int_{B_R}\frac{\alpha}{2}w^{\frac{1}{1-\theta}}|g|^2Ae^{Aw}\varphi\rho^{\alpha}v^{\beta+1}\,dx+\int_{B_R} |h|\varphi'\rho^\alpha v^\beta\,dx\\[5pt]
&&+\int_{B_R} \frac{\beta^2w^{-\frac{1}{1-\theta}}}{16(1+2e^{Aw})}|\nabla v|^2\varphi\rho^\alpha v^{\beta-3}\,dx+\int_{B_R} 4(1+2e^{Aw})A^2e^{2Aw}w^{\frac{1}{1-\theta}}|g|^2\varphi\rho^\alpha v^{\beta+1}\,dx\\[5pt]
&\leq&\int_{B_R} \frac{1}{2}w^{-\frac{3-2\theta}{1-\theta}}|\nabla w|^2\varphi \rho^\alpha v^{\beta-1}\,dx+\int_{B_R} \frac{1}{2}(A^2e^{Aw}+A^2e^{2Aw})^2w^{\frac{3-2\theta}{1-\theta}}|g|^2\varphi\rho^\alpha v^{\beta+1}\,dx\\[5pt]
&&+\int_{B_R} \frac{\alpha}{2}U^{ij}\rho_i\rho_j Ae^{Aw}\varphi\rho^{\alpha-2}v^{\beta}\,dx
+\int_{B_R}\frac{\alpha}{2}w^{\frac{1}{1-\theta}}|g|^2Ae^{Aw}\varphi\rho^{\alpha}v^{\beta+1}\,dx+\int_{B_R} |h|\varphi'\rho^\alpha v^\beta\,dx\\[5pt]
&&+\int_{B_R} \frac{\beta^2}{16(1+2e^{Aw})}U^{ij}v_iv_j\varphi\rho^\alpha v^{\beta-2}\,dx+\int_{B_R} 4(1+2e^{Aw})A^2e^{2Aw}w^{\frac{1}{1-\theta}}|g|^2\varphi\rho^\alpha v^{\beta+1}\,dx.
\end{eqnarray*}
Hence, \eqref{fm3} reduces to
\begin{eqnarray*}
0&\geq& \int_{B_R} \frac{\beta^2}{16(1+2e^{Aw})} U^{ij}v_iv_j \varphi \rho^{\alpha} v^{\beta-2}\,dx
+\alpha\int_{B_R} U^{ij}\rho_{ij}\varphi \rho^{\alpha-1}v^\beta\,dx-\int_{B_R} |h|\varphi'\rho^\alpha v^\beta\,dx\\[5pt]
&&+\int_{B_R} \left(\frac{\beta}{1-\theta}Awe^{Aw}-\frac{(\beta-1)^2}{3(1-\theta)^2}(1+2e^{Aw})-\frac{2\beta}{(1-\theta)^2}-
\frac{1}{2}\right)w^{-\frac{3-2\theta}{1-\theta}}|\nabla w|^2 \varphi\rho^\alpha v^{\beta-1}\,dx\\[5pt]
&&+\int_{B_R} \left(\alpha(\alpha-1)-(8+18e^{Aw}+\frac{1}{4}\beta)\alpha^2-\frac{\alpha}{2}Ae^{Aw}\right) U^{ij}\rho_i\rho_j\varphi\rho^{\alpha-2}v^\beta\,dx\\[5pt]
&&-\int_{B_R} \frac{1}{2}(A^2e^{Aw}+A^2e^{2Aw})^2w^{\frac{3-2\theta}{1-\theta}}|g|^2\varphi\rho^\alpha v^{\beta+1}\,dx-\int_{B_R}\frac{\alpha}{2}w^{\frac{1}{1-\theta}}|g|^2Ae^{Aw}\varphi\rho^{\alpha}v^{\beta+1}\,dx\\[5pt]
&&-\int_{B_R} 4(1+2e^{Aw})A^2e^{2Aw}w^{\frac{1}{1-\theta}}|g|^2\varphi\rho^\alpha v^{\beta+1}\,dx.
\end{eqnarray*}
Now we choose $A$ sufficiently large such that
\begin{eqnarray*}
\frac{\beta}{1-\theta}Awe^{Aw}-\frac{(\beta-1)^2}{3(1-\theta)^2}(1+2e^{Aw})-\frac{2\beta}{(1-\theta)^2}-\frac{1}{2}>0,
\end{eqnarray*}
i.e.,
\begin{eqnarray*}
(3\beta(1-\theta)Aw-2(\beta-1)^2)e^{Aw}-(\beta-1)^2-6\beta-\frac{3(1-\theta)^2}{2}>0.
\end{eqnarray*}
Note that
$$U^{ij}\rho_i\rho_j\leq v^{n-1}|\nabla\rho|^2\quad \text{and}\quad |U^{ij}\rho_{ij}|\leq 2nv^{n-1}.$$
Then we have
\begin{eqnarray*}
\int_{B_R} U^{ij} v_iv_j\rho^{\alpha}v^{\beta-2}\,dx\leq C\int_{B_R}\rho^\alpha v^{\beta+1}\,dx+C\int_{B_R} \rho^{\alpha-2}v^{\beta+n-1}\,dx+C\int_{B_R} |h|\rho^\alpha v^\beta\,dx.
\end{eqnarray*}
Since
\begin{eqnarray*}
U^{ij}(v^{\frac{\beta}{2}}\rho^{\frac{\alpha}{2}})_i(v^{\frac{\beta}{2}}\rho^{\frac{\alpha}{2}})_j&=&\frac{\beta^2}{4}U^{ij}v_iv_j\rho^{\alpha}v^{\beta-2}
+\frac{\alpha^2}{4}U^{ij}\rho_i\rho_j\rho^{\alpha-2}v^\beta+\frac{\alpha \beta}{2}U^{ij}v_i\rho_j\rho^{\alpha-1}v^{\beta-1}\\
&\leq&  \frac{\beta^2}{2}U^{ij}v_iv_j\rho^{\alpha}v^{\beta-2}
+\frac{\alpha^2}{2}U^{ij}\rho_i\rho_j\rho^{\alpha-2}v^\beta,
\end{eqnarray*}
we have
\begin{eqnarray*}
\int_{B_R} U^{ij} (v^{\frac{\beta}{2}}\rho^{\frac{\alpha}{2}})_i(v^{\frac{\beta}{2}}\rho^{\frac{\alpha}{2}})_j\,dx&\leq& C\int_{B_R} \rho^\alpha v^{\beta+1}\,dx+C\int_{B_R} \rho^{\alpha-2}v^{\beta+n-1}\,dx+C\int_{B_R} |h|\rho^\alpha v^\beta\,dx\\[5pt]
&\leq& C\int_{B_R} v^{\beta+n-1}\,dx+C\left(\int_{B_R}|h|^q\,dx\right)^{\frac{1}{q}}
\left(\int_{B_R}(v^{\frac{\beta}{2}}\rho^{\frac{\alpha}{2}})^{\frac{2q}{q-1}}\,dx\right)^{1-\frac{1}{q}}.
\end{eqnarray*}
Choosing $\chi\geq \frac{2q}{q-1}$ and using Lemma \ref{Sobolev} with $n=2$,
we have
\begin{eqnarray}
\left[\int_{B_R} (v^{\frac{\beta}{2}}\rho^{\frac{\alpha}{2}})^{\chi}\,dx\right]^{\frac{2}{\chi}}\leq C\int_{B_R} v^{\beta+1}\,dx+C.\label{fm7}
\end{eqnarray}
By the interior $W^{2,1+\varepsilon}$-estimates for Monge-Amp\`ere equation \cite{DFS, S} with \eqref{detcond-1}, we know that there is a small $\varepsilon_0>0$, that depends only on $\lambda$ and $\Lambda$ such that $\|v\|_{L^{1+\varepsilon_0}}\leq C(\lambda,\Lambda)$. Then in \eqref{fm7}, we choose $\beta=\varepsilon_0$ and $\chi=\frac{2p}{\varepsilon_0}$ for $p\in(\frac{q}{q-1}\varepsilon_0,+\infty)$ to get
$\|\Delta u\|_{L^p(B_{R/2})}\leq C$.

For $\theta=1$, we know that $w=\ln\det D^2u$, and we can obtain the $L^p$-bound of $\Delta u$ following the same method used above. Alternatively, we can apply tha Legendre transform to $u$ for the case $\theta=1$ to get the estimate of $D^2u$ by the strictly convexity of $u$ with condition \eqref{detcond-1}.
\end{proof}

%\subsection{Higher estimates}\label{C2-bound}
%In this subsection, we will prove the higher estimates of $u$ in terms of the $W^{2,p}$-estimate.

%\begin{proof}[Proof of Theorem \ref{int-est-inh}]
Now we establish the higher estimates of $u$ in terms of the $W^{2,p}$-estimate.
By chaining together a sequence of balls, in a standard fashion, we know $\Delta u\in L^p_{loc}(\Omega)$ for any fixed $p\in[1,+\infty)$. Then by Theorem \ref{thm4.2} and the same arguments as \cite[Proposition 6.1]{Lo} or \cite[Theorem 1.3]{Le3},  we get the H\"older continuity of $w$ and all the higher order estimates of $u$.

%, where they used Caffarelli-Guti\'errez's Harnack inequality for homogeneous linearized Monge-Amp\`ere equation, local and global boundedness of inhomogeneous linearized Monge-Amp\`ere equation, which can be viewed as a degenerate elliptic equation of divergence form to get the H\"older continuity of $w$, then higher order estimates of $u$ followed by Monge-Amp\`ere equation theory and classical $L^p$ theory, hence we can complete the proof.

%\begin{rem}
%Note that the condition $\lambda^{-1},\,\Lambda\in L^p_{loc}(\Omega)$ is not enough to guarantee H\"older continuity of solutions to degenerate elliptic equations directly, see \cite{Tr1}, that's the reason we should use Caffarelli-Guti\'errez's Harnack inequality in the proof of Theorem \ref{int-est-inh} .
%\end{rem}

As we mentioned in the introduction, there is another approach to establish the $C^2$-estimate from $W^{2,p}$ when $f\in W^{1,q}(B_1)$, which holds in any dimension.
\begin{theo}\label{C2-est}
Let $u$ be a uniformly convex smooth solution to equation \eqref{4-eq-inh} in $B_1$ with $\theta=0$ or $1$. Assume that $f\in W^{1,q}(B_1)$ for some $q>\frac{n}{2}$. Suppose $u$ satisfies
\eqref{detcond-1}. Assume that $\|\Delta u\|_{L^{p_n}(B_1)}$ is bounded for some $p_n>\frac{n(n-1)}{2}$. Then there exists a constant $C>0$ depending on $\lambda$, $\Lambda$, $\|\Delta u\|_{L^{p_n}(B_1)}$ and $\|f\|_{W^{1,q}(B_1)}$, such that
$$\sup_{B_{1/2}}\Delta u\leq C.$$
\end{theo}
\begin{proof}
The proof is inspired by \cite{CC}.
We consider the case $\theta=1$. The case $\theta=0$ then follows by using the Legendre transform.
Denote $v=u^{kl}w_kw_l$, where $\{u^{kl}\}$ is the inverse matrix of $D^2u$.  By direct calculations,
\begin{eqnarray*}
v_i&=&-u^{ks}u^{tl}u_{sti}w_kw_l+2u^{kl}w_{ik}w_l,\label{v_i}\\[5pt]
v_{ij}&=&u^{kp}u^{rs}u^{tl}u_{prj}u_{sti}w_kw_l+u^{ks}u^{lp}u^{rt}u_{prj}u_{sti}w_kw_l-u^{ks}u^{tl}u_{stij}w_kw_l
\nonumber\\
&&-2u^{ks}u^{tl}u_{sti}w_{jk}w_l-2u^{ks}u^{tl}u_{stj}w_{ik}w_l+2u^{kl}w_{ijk}w_l+2u^{kl}w_{ik}w_{jl}.\label{v_ij}
\end{eqnarray*}
For any $x_0\in B_1$, we can choose a coordinate transformation so that $u_{ij}(x_0)=u_{ii}(x_0)\delta_{ij}$. Then
\begin{eqnarray}
u^{ij}v_{ij}(x_0)&=&2u^{ii}u^{kk}u^{rr}u^{ll}u_{kri}u_{lri}w_kw_l-u^{ii}u^{kk}u^{ll}u_{iikl}w_kw_l-4u^{ii}u^{kk}u^{ll}u_{kli}w_{ik}w_l
\nonumber\\
&&+2u^{ii}u^{kk}w_{iik}w_k+2u^{ii}u^{kk}w_{ik}^2.\label{u^ijv_ij}
\end{eqnarray}
Differentiating $\ln\det D^2u=w$ twice, we have
\begin{eqnarray}
u^{ij}u_{ijkl}-u^{is}u^{tj}u_{ijk}u_{stl}=(\ln\det D^2u)_{kl}=w_{kl}.\label{G_kl}
\end{eqnarray}
Note that \eqref{4-eq-inh} can be written as $u^{ij}w_{ij}=e^{-w}f$. Differentiating the equation respect to $x_k$-direction directly yields
\begin{eqnarray}
-u^{ip}u^{rj}u_{prk}w_{ij}+u^{ij}w_{ijk}=e^{-w}(f_k-w_kf).\label{g_k}
\end{eqnarray}
Inserting \eqref{G_kl} and \eqref{g_k} into  \eqref{u^ijv_ij}, we have
\begin{eqnarray}
u^{ij}v_{ij}(x_0)
&=&u^{kk}u^{ll}u^{ii}u^{jj}u_{kji} u_{lji}w_kw_l-2u^{ii}u^{kk}u^{ll}u_{kli}w_{ik}w_l+u^{ii}u^{kk}w_{ik}^2\nonumber\\[5pt]
&&+u^{ii}u^{kk}w_{ik}w_{ik}-u^{kk}u^{ll}w_{kl}w_kw_l+2u^{kk}e^{-w}(f_k-w_kf)w_k\nonumber\\[5pt]
&=&u^{ii}u^{kk}\left(\sum_lu^{ll}u_{ikl}w_l\right)^2-2u^{ii}u^{kk}\left(\sum_lu^{ll}u_{ikl}w_l\right)w_{ik}+u^{ii}u^{kk}w_{ik}^2\nonumber\\[5pt]
&&+u^{ii}u^{kk}w_{ik}^2-u^{kk}u^{ll}w_{kl}w_kw_l+2u^{kk}e^{-w}(f_k-w_kf)w_k\nonumber\\[5pt]
&=&u^{ii}u^{kk}|w_{ik}-\sum_lu^{ll}u_{ikl}w_l|^2+u^{ii}u^{kk}w_{ik}^2-u^{kk}u^{ll}w_{kl}w_kw_l+2u^{kk}e^{-w}(f_k-w_kf)w_k.\label{u^ijv_ij-1}
\end{eqnarray}
Next, we compute
\begin{eqnarray}
u^{ij}(e^{\frac{1}{2}w}v)_{ij}(x_0)&=&\frac{1}{4}e^{\frac{1}{2}w}u^{ii}w_i^2u^{kk}w_k^2
+\frac{1}{2}e^{\frac{1}{2}w}u^{ii}w_{ii}v+e^{\frac{1}{2}w}u^{ii}w_iv_i
+e^{\frac{1}{2}w}u^{ii}v_{ii}.\label{wv}
\end{eqnarray}
Note that
\begin{eqnarray}
u^{ii}w_iv_i&=&-u^{ii}u^{kk}u^{ll}u_{ikl}w_kw_lw_i+2u^{ii}u^{kk}w_{ki}w_kw_i\nonumber\\
&=&u^{ii}u^{kk}w_kw_i(w_{ki}-\sum_l u^{ll}u_{ikl}w_l)+u^{ii}u^{kk}w_{ki}w_kw_i.\label{w_iG_i}
\end{eqnarray}
Combining \eqref{u^ijv_ij-1}, \eqref{wv} and \eqref{w_iG_i}, we have
\begin{eqnarray}
u^{ij}(e^{\frac{1}{2}w}v)_{ij}(x_0)&=&u^{ii}u^{kk}|w_{ik}-\sum_lu^{ll}u_{ikl}w_l+\frac{1}{2}w_iw_k|^2e^{\frac{1}{2}w}+u^{ii}u^{kk}w_{ik}^2e^{\frac{1}{2}w}
\nonumber\\&&-\frac{3}{2}e^{-\frac{1}{2}w}fv+2e^{-\frac{1}{2}w}u^{kk}f_kw_k.\label{u^ijv_ije}
\end{eqnarray}
By \eqref{G_kl}, we have
\begin{eqnarray}
u^{ij}(\Delta u)_{ij}(x_0)=u^{ii}u^{jj}u_{ijk}^2+\Delta w.\label{u_ii}
\end{eqnarray}
Let $z=e^{\frac{1}{2}w}v+\Delta u$.
Note that
$$\Delta w=\sum_{k}w_{kk}\leq\frac{1}{2}(u^{kk})^2w_{kk}^2e^{\frac{1}{2}w}+\frac{1}{2}u_{kk}^2e^{-\frac{1}{2}w}
\leq\frac{1}{2}(u^{kk})^2w_{kk}^2e^{\frac{1}{2}w}+\frac{1}{2}(\Delta u)^2e^{-\frac{1}{2}w}.$$
Then by \eqref{u^ijv_ije} and \eqref{u_ii}, we have
\begin{eqnarray}
u^{ij}z_{ij}(x_0)&\geq& -C|f|v-C|f_k|v-C|f_k|+u^{kk}u^{ll}w_{kl}^2e^{\frac{1}{2}w}+\Delta w\nonumber\\
&\geq&-C(|f|+|f_k|)v-C|f_k|+(u^{kk})^2w_{kk}^2e^{\frac{1}{2}w}-\frac{1}{2}(u^{kk})^2w_{kk}^2e^{\frac{1}{2}w}-\frac{1}{2}(\Delta u)^2e^{-\frac{1}{2}w}\nonumber\\
&\geq&-C(|f|+|f_k|+\Delta u)z,\label{u_ijz_ij}
\end{eqnarray}
where $C=C(\lambda,\Lambda)$. In the last inequality we used $z\geq \Delta u\geq n(\det D^2u)^{\frac{1}{n}}\geq n\lambda^{\frac{1}{n}}$.
Since \eqref{u_ijz_ij} is valid at every point in $B_1$, by \eqref{detcond-1}, we have the following inequality
\beq\label{sub-ineq}
D_j(U^{ij}D_iz)\geq-gz-C\Delta uz\quad \text{in}\,\, B_1,
\eeq
where $g=C(|f|+|f_k|)\in L^{q}(B_1)$.

Next, we drive the upper bound of $z$ by integration and iteration. Let $\eta\in C_0^\infty(B_1)$ be a cutoff function. Multiplying  \eqref{sub-ineq} by
$\varphi=\eta^2z^{\beta-1}$  with $\beta\geq 2$ and by integration by parts, we have
\begin{eqnarray*}
(\beta-1)\int_{B_1}U^{ij}z_iz_j\eta^2z^{\beta-2}\,dx&\leq& -2\int_{B_1}U^{ij}z_i\eta_j\eta z^{\beta-1}\,dx+C\int_{B_1}\Delta u\cdot \eta^2 z^{\beta}\,dx+\int_{B_1}g\eta^2z^{\beta}\,dx\\[5pt]
&\leq&\frac{\beta-1}{2}\int_{B_1} U^{ij}z_iz_j\eta^2z^{\beta-2}\,dx+\frac{2}{\beta-1}\int_{B_1}U^{ij}\eta_i\eta_j z^{\beta}\,dx\\
&&+C\int_{B_1}\Delta u\cdot \eta^2 z^{\beta}\,dx+\int_{B_1}g\eta^2z^{\beta}\,dx,
\end{eqnarray*}
which implies
\begin{eqnarray*}
\int_{B_1} U^{ij}(\eta z^{\frac{\beta}{2}})_i(\eta z^{\frac{\beta}{2}})_j\,dx\leq C\beta\left(\int_{B_1}U^{ij}\eta_i\eta_j z^{\beta}\,dx+\int_{B_1}\Delta u\cdot \eta^2 z^{\beta}\,dx+\int_{B_1}g\eta^2z^{\beta}\,dx\right).
\end{eqnarray*}
Then by the Monge-Amp\`ere Sobolev inequality \cite{TW} and $(U^{ij})\leq (\Delta u)^{n-1}I$, we have
\begin{eqnarray}
\|\eta z^{\frac{\beta}{2}}\|_{L^{p}(B_1)}^2\leq C_0\beta\left(\int_{B_1}(\Delta u)^{n-1}|D\eta|^2 z^{\beta}\,dx+\int_{B_1}\Delta u\cdot \eta^2 z^{\beta}\,dx+\int_{B_1}g\eta^2z^{\beta}\,dx\right),\label{re-hol}
\end{eqnarray}
where $p=2^*=\frac{2n}{n-2}$ for $n>2$ and $p>2$ for $n=2$. Then by H\"older's inequality
$$\int_{B_1} g\eta^2 z^{\beta}\,dx\leq\left(\int_{B_1}g^q\,dx\right)^{\frac{1}{q}}\left(\int_{B_1}|\eta z^{\frac{\beta}{2}}|^{\frac{2q}{q-1}}\right)^{1-\frac{1}{q}}.$$
Since $q>\frac{n}{2}$, we have $2<\frac{2q}{q-1}<2^*$.
By interpolation inequality, we obtain
$$\|\eta z^{\frac{\beta}{2}}\|_{L^{\frac{2q}{q-1}}(B_1)}\leq \varepsilon \|\eta z^{\frac{\beta}{2}}\|_{L^{p}(B_1)}+C(n,q)\varepsilon^{-\frac{n}{2q-n}}\|\eta z^{\frac{\beta}{2}}\|_{L^2(B_1)}.$$
Then we choose $\varepsilon=\left(4C_0\|g\|_{L^q(B_1)}\beta+1\right)^{-\frac{1}{2}}$.  By \eqref{re-hol},
\begin{eqnarray}
\|\eta z^{\frac{\beta}{2}}\|_{L^{p}(B_1)}^2\leq C\beta^\alpha\left(\int_{B_1}(\Delta u)^{n-1}|D\eta|^2 z^{\beta}\,dx+\int_{B_1}\Delta u\cdot \eta^2 z^{\beta}\,dx+\int_{B_1}\eta^2z^{\beta}\,dx\right),\label{fm8}
\end{eqnarray}
where $\alpha=\frac{2q}{2q-n}$. Then by H\"older's inequality, we have
\begin{eqnarray}
\int_{B_1}(\Delta u)^{n-1}|D\eta|^2 z^{\beta}\,dx&\leq&\|\Delta u\|_{L^{p_n}(B_1)}\cdot\| |D\eta| z^{\frac{\beta}{2}}\|_{L^{\frac{2p_n}{p_n-n+1}}(B_1)}^2,\label{fm9}\\[5pt]
\int_{B_1}\Delta u\cdot \eta^2 z^{\beta}\,dx&\leq&\|\Delta u\|_{L^{p_n}(B_1)}\cdot \|\eta z^{\frac{\beta}{2}}\|_{L^{\frac{2p_n}{p_n-1}}(B_1)}^2.\label{fm10}
\end{eqnarray}
Combining \eqref{fm8}, \eqref{fm9} and \eqref{fm10}, we get
\begin{eqnarray*}
\|\eta z^{\frac{\beta}{2}}\|_{L^{p}(B_1)}&\leq& C\beta^{\alpha/2}\left(\| |D\eta| z^{\frac{\beta}{2}}\|_{L^{\frac{2p_n}{p_n-n+1}}(B_1)}+\|\eta z^{\frac{\beta}{2}}\|_{L^{\frac{2p_n}{p_n-1}}(B_1)}+\|\eta z^{\frac{\beta}{2}}\|_{L^2(B_1)}\right)\\[5pt]
&\leq& C\beta^{\alpha/2}\left(\| |D\eta| z^{\frac{\beta}{2}}\|_{L^{\frac{2p_n}{p_n-n+1}}(B_1)}+\|\eta z^{\frac{\beta}{2}}\|_{L^{\frac{2p_n}{p_n-n+1}}(B_1)}\right).
\end{eqnarray*}
Now for any $0<r<R\leq 1$, we choose a cutoff function $\eta\in C_0^\infty(B_R)$ such that
$$0\leq\eta\leq 1,\quad\eta\equiv 1\,\,\text{in}\,\, B_r\quad\text{and}\quad |D\eta|\leq \frac{2}{R-r}.$$
Then we obtain
$$\|z^{\frac{\beta}{2}}\|_{L^{p}(B_r)}\leq \frac{C\beta^{\frac{\alpha}{2}}}{R-r}\|z^{\frac{\beta}{2}}\|_{L^{\frac{2p_n}{p_n-n+1}}(B_R)}.$$
By the assumption $p_n>\frac{n(n-1)}{2}$, we know $p>\frac{2p_n}{p_n-n+1}$. Denote $\chi=p\frac{p_n-n+1}{2p_n}>1$. Then
\begin{eqnarray}
\|z\|_{L^{\frac{\beta p_n}{p_n-n+1}\chi}(B_r)}\leq \frac{C^{\frac{2}{\beta}}\beta^{\frac{\alpha}{\beta}}}{(R-r)^{\frac{2}{\beta}}}\|z\|_{L^{\frac{\beta p_n}{p_n-n+1}}(B_R)}.\label{iter}
\end{eqnarray}
We iterate \eqref{iter} to get the desired estimate. Set
$$\beta_i=2\chi^{i}\quad\text{and}\quad R_i=r+\frac{R-r}{2^{i}},\  i=0,1,2,\cdots,$$
i.e.,
$$\beta_i=\chi\beta_{i-1}\ \text{and} \ R_{i-1}-R_i=\frac{R-r}{2^{i}},\  i=1,2,\cdots.$$
By \eqref{iter},
$$\|z\|_{L^{\frac{2p_n}{p_n-n+1}\chi^{i+1}}(B_{R_{i+1}})}
\leq C^{\sum_{j=0}^i\frac{2}{\beta_j}}\cdot\prod_{j=0}^i \beta_j^{\frac{\alpha}{\beta_j}}\cdot 4^{\sum_{j=0}^i\frac{j}{\beta_j}}\frac{1}{(R-r)^{\sum_{j=0}^i\frac{2}{\beta_j}}}\cdot\|z\|_{L^{\frac{2p_n}{p_n-n+1}}(B_R)}.$$
Letting $i\to\infty$, by Young's inequality, we have
$$\begin{aligned}
\|z\|_{L^\infty(B_r)}&\leq \frac{C}{(R-r)^{\frac{\chi}{\chi-1}}}\|z\|_{L^{\frac{2p_n}{p_n-n+1}}(B_R)}\\[4pt]
&=\frac{C}{(R-r)^{\frac{\chi}{\chi-1}}}\|z\|_{L^{1}(B_R)}^{\frac{p_n-n+1}{2p_n}}\cdot\|z\|_{L^{\infty}(B_R)}^{\frac{p_n+n-1}{2p_n}}\\[4pt]
&\leq \frac{1}{2}\|z\|_{L^{\infty}(B_R)}+\frac{C}{(R-r)^{\frac{\chi}{\chi-1}\cdot\frac{2p_n}{p_n-n+1}}}\|z\|_{L^1(B_R)}.
\end{aligned}$$
Set $f(t)=\|z\|_{L^\infty(B_t)}$ for $t\in (0,1].$ Then for any $0<r<R\leq \overline{R}<1$,
$$f(r)\leq \frac{1}{2}f(R)+\frac{C}{(R-r)^{\frac{\chi}{\chi-1}\cdot\frac{2p_n}{p_n-n+1}}}\|z\|_{L^1(B_{\overline{R}})}.$$
We apply Lemma \ref{lem4.3} below to get
\begin{eqnarray}
f(r)\leq \frac{C}{(R-r)^{\frac{\chi}{\chi-1}\cdot\frac{2p_n}{p_n-n+1}}}\|z\|_{L^1(B_{\overline{R}})}.\label{f(r)}
\end{eqnarray}
It remains to show $\|z\|_{L^1(B_{\overline{R}})}\leq C$. It is clear that $\Delta u\in L^1(B_{\overline{R}})$,
hence it is suffices to estimate the integral of $v$.
Let $\eta\in C_0^\infty(B_1)$  be a cutoff function such that $\eta\equiv 1$ in ${B_{\overline{R}}}$.
Multiplying \eqref{4-eq-inh} by  $\varphi=\eta^2w$ and integrating by parts, we have
$$\int_{B_1}U^{ij}w_iw_j\eta^2\,dx+2\int_{B_1}U^{ij}w_i\eta_j\eta w\,dx=\int_{B_1}fw\eta^2\,dx.$$
Then by the Cauchy inequality, we get
\begin{eqnarray*}
\frac{1}{2}\int_{B_1}U^{ij}w_iw_j\eta^2\,dx&\leq& 2\int_{B_1}U^{ij}\eta_i\eta_j w^2\,dx+\int_{B_1}|f|w\eta^2\,dx\\[5pt]
&\leq&2\int_{B_1}(\Delta u)^{n-1}|D\eta|^2w^2\,dx+\int_{B_1}|f|w\eta^2\,dx.
\end{eqnarray*}
Hence $\int_{B_{\overline{R}}}u^{kl}w_kw_l\,dx\leq C$ follows by $\Delta u\in L^{p_n}(B_1)$, $f\in L^q(B_1)$.
Then we complete the proof by choosing $r=\frac{1}{2}$ and $R=\overline{R}$ in \eqref{f(r)}.
\end{proof}
\begin{lem}[{\cite[Lemma 4.3]{HL}}]\label{hma1}\label{lem4.3}
Let $f(t)\geq 0$ be bounded in $[\tau_0,\tau_1]$ with $\tau_0\geq 0$. Suppose for $\tau_0\leq t<s\leq \tau_1$ we have
$$f(t)\leq \beta f(s)+\frac{A}{(s-t)^\alpha}+B$$
for some $\beta\in[0,1)$. Then for any $\tau_0\leq t<s\leq\tau_1$ there holds
$$f(t)\leq C(\alpha,\beta)\left\{\frac{A}{(s-t)^\alpha}+B\right\}.$$
\end{lem}

\begin{rem}
 Note that for $n=2$, we have already got an $L^p$-bound of $\Delta u$ in Theorem \ref{W2p-est} for all $p\in[1,+\infty)$. However, it is still unknown how to obtain $L^p$-bound of $\Delta u$ when $n\geq 3$ by the integral method used in Theorem \ref{W2p-est}. We also do not know how to obtain higher estimates for other exponents $\theta$ by the above method.
\end{rem}

%In the following, we will state some results related to linearized Monge-Amp\`ere equations and degenerate elliptic equations, which originally came from \cite{CG}, \cite{MS} and \cite{Tr}, but stated in \cite{Lo} with different versions.
%
%\begin{theo}{\cite[Theorem 3.4]{Lo}}\label{harnack}
%Let $\Omega$ be a domain in $\mathbb R^n$, let $u$ be a smooth convex function in $\Omega$ satisfies \eqref{detcond-1}. Let $w$ be a solution in $\Omega$ of the linearized homogeneous Monge-Amp\`ere equation
%$$U^{ij}w_{ij}=0$$
%where $U^{ij}$ is the co-matrix of $D^2u$, let $R>0$ and $y\in\Omega$ be such that $B_R(y)\subset\Omega$, then for some $\beta<1$ depending only on $\lambda$, $\Lambda$, for any $r<R/4$,
%$$\operatorname{osc}(r/2)\leq\beta\operatorname{osc}(r),$$
%where
%\begin{eqnarray*}
%\operatorname{osc}(r)=\sup_{B_r(y)}w-\inf_{B_r(y)}w.
%\end{eqnarray*}
%\end{theo}

\vskip 20pt

\section{An interior estimate in higher dimensions}\label{maindimh}

In this section, we will prove Theorem \ref{main-thm}. By  \cite{CG, Ca}, it suffices to get the interior estimates on the upper and lower bound of $\det D^2u$.
 Note that there exists a constant $c_{n}>0$ such that
\begin{equation}\label{elliptic}
\frac{\det D^2u}{c_{n}\cdot\Delta u} I \leq U^ {i j} \leq c_{n}(\Delta u)^{n-1} I.
\end{equation}
In view of \eqref{div2} and \eqref{elliptic}, we first consider the following degenerate linear elliptic equation
\beq\label{de-eq1}
-D_j(a_{ij}(x)D_i u)+c(x)u=f(x)\quad\text{in}\,\,\Omega.
\eeq
\begin{lem}\label{est}
Assume $\{a_{ij}(x)\}$ satisfies
\beq\label{elli}
\frac{d(x)}{\lambda(x)}\leq a_{ij}(x)\leq \lambda(x)^{n-1}\quad\text{in}\,\,\,B_1,
\eeq
where $\lambda(x)\in L^p(B_1)$, $d^{-1}\in L^q(B_1)$. Assume $c(x), f(x)\in L^{p_0}(B_1)$.
Let $u\in W^{1,p}(B_1)$ be a subsolution in the following sense:
\begin{equation}\label{eq}
\int_{B_1} a_{ij}D_{i}u D_{j}\varphi+cu\varphi\, dx\leq \int_{B_1} f\varphi\,dx\quad \text{for\,\,any}\,\, \varphi\in W^{1,p}_0(B_1)\text{\ and}\,\,\varphi\geq 0.
\end{equation}
Suppose $p$, $q$ and $p_0$ satisfy $\frac{n}{p}+\frac{1}{q}<\frac{2}{n}$ and $p_0\geq\frac{p}{n-1}$. Then
$$\sup_{B_{1/2}}u\leq C\left(\|u^+\|_{L^1(B_1)}+\|f\|_{L^{p_0}(B_1)}\right),$$
where $C$ depends only on $p$, $\|\lambda\|_{L^p(B_1)}$, $\|d^{-1}\|_{L^q(B_1)}$ and $\|c\|_{L^{p_0}(B_1)}$.%, $\|c\|_{L^r(B_1)}$ and $\|f\|_{L^r(B_1)}$
\end{lem}

\begin{proof}
For some $k>0$ and $m>0$, set $\overline{u}=u^++k$ and
$$\overline{u}_m=\left\{\begin{array}{lc}\overline{u},&\ u<m,\\k+m,&\ u\geq m.\end{array}\right.$$
Let $\eta\in C^\infty_0(B_1)$. We choose a test function
$\varphi=\eta^2(\overline{u}_m^\beta \overline{u}-k^{\beta+1})\in W^{1,p}_0(B_1)$
for some $\beta\geq 0$ to be determined later. Substituting $\varphi$ into \eqref{eq}, we have
\begin{eqnarray*}
&&\beta\int_{B_1}a_{ij}D_i\overline{u} D_j\overline{u}_m \overline{u}_m^{\beta-1} \overline{u}\eta^2\, dx+\int_{B_1}a_{ij}D_i\overline{u} D_j\overline{u}\,\overline{u}_m^\beta \eta^2\, dx\\[4pt]
&&\ \ \ \ \ \ \ \leq -2\int_{B_1} a_{ij} D_i\overline{u} D_j\eta  \overline{u}_m^\beta \overline{u}  \eta \, dx+\int_{B_1} \left(|c|\eta^2 \overline{u}_m^\beta \overline{u}^2+|f|\eta^2 \overline{u}_m^\beta \overline{u}\right)\, dx.
\end{eqnarray*}
Note that $D\overline{u}=D\overline{u}_m$ in $\{u<m\}$ and $D\overline{u}_m=0$ in $\{u\geq m\}$. By the Cauchy inequality, we have
$$
\begin{aligned}
\beta&\int_{B_1}a_{ij}D_i\overline{u}_m D_j\overline{u}_m \overline{u}_m^{\beta-1} \overline{u}\eta^2\, dx+\int_{B_1}a_{ij}D_i\overline{u} D_j\overline{u}\, \overline{u}_m^\beta \eta^2\, dx\\[4pt]
&\leq \frac{1}{2}\int_{B_1} a_{ij} D_i\overline{u} D_j\overline{u} \, \overline{u}_m^\beta \eta^2\, dx+4\int_{B_1} a_{ij}D_i\eta D_j \eta \overline{u}_m^\beta \overline{u}^2\, dx+\int_{B_1} \left(|c|\eta^2 \overline{u}_m^\beta \overline{u}^2+|f|\eta^2 \overline{u}_m^\beta \overline{u}\right)\, dx.
\end{aligned}
$$
Then by \eqref{elli}
\begin{eqnarray}
&&\beta\int_{B_1}\frac{d}{\lambda(x)}|D\overline{u}_m|^2 \overline{u}_m^{\beta}\eta^2\, dx+\frac{1}{2}\int_{B_1}\frac{d}{\lambda(x)}|D\overline{u}|^2\overline{u}_m^{\beta}\eta^2\, dx\nonumber\\[4pt]
&\leq& \beta\int_{B_1}a_{ij}D_i\overline{u}_m D_j\overline{u}_m \overline{u}_m^{\beta-1} \overline{u}\eta^2\, dx+\frac{1}{2}\int_{B_1}a_{ij}D_i\overline{u} D_j\overline{u}\, \overline{u}_m^\beta \eta^2\, dx \nonumber\\[4pt]
&\leq& 4 \int_{B_1}\lambda(x)^{n-1}|D\eta|^2 \overline{u}_m^{\beta}\overline{u}^2\, dx+\int_{B_1}c_0\eta^2 \overline{u}_m^\beta \overline{u}^2\, dx,\label{a1}
\end{eqnarray}
where $c_0=|c|+\frac{|f|}{k}$.
Choose $k=\|f\|_{L^{p_0}(B_1)}$ if $f$ is not identically 0. Otherwise choose arbitrary $k>0$ and let $k\to 0^+$.
Let $w=\overline{u}_m^{\frac{\beta}{2}}\overline{u}$.  There holds
$$|Dw|^2=\left|\frac{\beta}{2}\overline{u}_m^{\frac{\beta}{2}-1}\overline{u}_m D\overline{u}_m+\overline{u}_m^{\frac{\beta}{2}}D\overline{u}\right|^2\leq (1+\beta)(\beta \overline{u}_m^{\beta}|D\overline{u}_m|^2+\overline{u}_m^\beta|D\overline{u}|^2).$$
Therefore by \eqref{a1}
$$\int_{B_1}\frac{d}{\lambda(x)}|Dw|^2\eta^2\, dx\leq 8(1+\beta)\int_{B_1} \lambda(x)^{n-1}|D\eta|^2 w^2\, dx+2(1+\beta)\int_{B_1}c_0w^2\eta^2\, dx.$$
By $|D(w\eta)|^2\leq 2|Dw|^2\eta^2+2|D\eta|^2 w^2$,
we have
\begin{eqnarray}
\int_{B_1} \frac{d}{\lambda(x)}|D(w\eta)|^2\, dx&\leq&\int_{B_1} \left(\frac{2d}{\lambda(x)}+16(1+\beta)\lambda(x)^{n-1}\right)|D\eta|^2w^2\, dx+4(1+\beta)\int_{B_1}c_0w^2\eta^2\, dx\nonumber\\[4pt]
&\leq &18(1+\beta)\int_{B_1}\lambda(x)^{n-1}|D\eta|^2w^2\, dx+4(1+\beta)\int_{B_1}c_0w^2\eta^2\, dx.\label{est-1}
\end{eqnarray}
Next, we deal with the $\lambda(x)$ in the above estimate. By the assumptions $\lambda\in L^p(B_1)$ and $d^{-1}\in L^q(B_1)$, we have $\lambda d^{-1}\in L^{\frac{pq}{p+q}}(B_1)$.
By H{\"o}lder's inequality, we have
\begin{eqnarray}
\left\|D(w\eta)\right\|_{L^{\frac{2pq}{pq+p+q}}(B_1)}^{2} &\leq&\|\lambda d^{-1}\|_{L^{\frac{pq}{p+q}}(B_1)} \cdot \int_{B_{1}} \frac{1}{\lambda(x)d^{-1}}\left|D(w \eta)\right|^{2}\,dx,\label{norm-est}\\[3pt]
\int_{B_{1}} \lambda(x)^{n-1}|D\eta|^2w^2\,dx&\leq&\|\lambda\|_{L^p(B_1)}\cdot\|wD\eta\|_{L^{\frac{2p}{p-n+1}}(B_1)}^2,\label{norm-est-1}\\[3pt]
\int_{B_1}c_0w^2\eta^2\,dx&\leq&\|c_0\|_{L^{p_0}(B_1)}\cdot\left(\int_{B_1}(\eta w)^{\frac{2{p_0}}{{p_0}-1}}\,dx\right)^{1-\frac{1}{{p_0}}}.\label{norm-est-3}
\end{eqnarray}
Combining \eqref{est-1}-\eqref{norm-est-3},  we get
%$$\left\|D(w\eta)\right\|_{L^{2-\varepsilon}}^{2} \leq 18(1+\beta)\|\lambda\|_{L^\frac{2}{\varepsilon}-1}^2\|wD\eta\|_{L^{\frac{2-\varepsilon}{1-\varepsilon}}}^2.$$
%Next, we want to let $\frac{2}{\varepsilon}-1=p$, we obtain $\varepsilon=\frac{2}{p+1}.$ With this choice, we have $2-\varepsilon=\frac{2p}{p+1}$ and $\frac{2-\varepsilon}{1-\varepsilon}=\frac{2p}{p-1}.$ Hence, we know
\begin{eqnarray*}
\|D(w\eta)\|_{L^{\frac{2pq}{pq+p+q}}(B_1)}&\leq& C(1+\beta)^{\frac{1}{2}}\left(\|wD\eta\|_{L^{\frac{2p}{p-n+1}}(B_1)}+\|w\eta\|_{L^{\frac{2{p_0}}{{p_0}-1}}(B_1)}\right)\\[4pt]
&\leq& C(1+\beta)^{\frac{1}{2}}\left(\|wD\eta\|_{L^{\frac{2p}{p-n+1}}(B_1)}+\|w\eta\|_{L^{\frac{2p}{p-n+1}}(B_1)}\right)
\end{eqnarray*}
for $C$ depending on $\|\lambda\|_{L^p(B_1)}$, $\|d^{-1}\|_{L^q(B_1)}$, $\|c\|_{L^{p_0}(B_1)}$. Here we used $p_0\geq\frac{n-1}{p}$.
By the Sobolev inequality
\begin{equation*}%\label{re-Holder}
\|w\eta\|_{L^\alpha(B_1)}\leq C(1+\beta)^{\frac{1}{2}}\left(\|wD\eta\|_{L^{\frac{2p}{p-n+1}}(B_1)}+\|w\eta\|_{L^{\frac{2p}{p-n+1}}(B_1)}\right),
\end{equation*}
where
\beq\label{alp}
\frac{1}{\alpha}=\frac{pq+p+q}{2pq}-\frac{1}{n}.
\eeq
Now for any $0<r<R\leq 1$, we choose a cutoff function $\eta\in C_0^\infty(B_R)$ such that
$$0\leq\eta\leq 1,\quad\eta\equiv 1\,\,\text{in}\,\, B_r\quad\text{and}\quad |D\eta|\leq \frac{2}{R-r}.$$
Then we obtain
$$\|w\|_{L^\alpha(B_r)}\leq \frac{C(1+\beta)^{\frac{1}{2}}}{R-r}\|w\|_{L^{\frac{2p}{p-n+1}}(B_R)}.$$
By \eqref{alp} and the assumption $\frac{n}{p}+\frac{1}{q}<\frac{2}{n}$,
we have $\alpha>\frac{2p}{p-n+1}$. We can do the iteration as follows.

Recalling the definition of $w$, we have
$$\|\overline{u}_m^{\frac{\beta}{2}}\overline{u}\|_{L^\alpha(B_r)}\leq\frac{C(1+\beta)^{\frac{1}{2}}}{R-r}\|\overline{u}_m^{\frac{\beta}{2}}\overline{u}\|_{L^{\frac{2p}{p-n+1}}(B_R)}.$$
Set $\gamma=\beta+2\geq 2$. By $\overline{u}_m\leq \overline{u}$, we obtain
$$\|\overline{u}_m^{\frac{\gamma}{2}}\|_{L^\alpha(B_r)}\leq\frac{C\gamma^{\frac{1}{2}}}{R-r}\|\overline{u}^{\frac{\gamma}{2}}\|_{L^{\frac{2p}{p-n+1}}(B_R)}.$$
Letting $m\to\infty$, we get
$$\|\overline{u}^{\frac{\gamma}{2}}\|_{L^\alpha(B_r)}\leq\frac{C\gamma^{\frac{1}{2}}}{R-r}\|\overline{u}^{\frac{\gamma}{2}}\|_{L^{\frac{2p}{p-n+1}}(B_R)},$$
i.e.,
$$\|\overline{u}\|_{L^{{\frac{\gamma}{2}}\alpha}(B_r)}\leq\frac{(C\gamma)^{\frac{1}{\gamma}}}{(R-r)^{\frac{2}{\gamma}}}\|\overline{u}\|_{L^{{\frac{\gamma}{2}}\frac{2p}{p-n+1}}(B_R)}.$$
Denote $\chi=\alpha\cdot\frac{p-n+1}{2p}>1$. Then
\begin{equation}\label{itt}
\|\overline{u}\|_{L^{{\frac{{\gamma}p}{p-n+1}\chi}}(B_r)}\leq\frac{(C\gamma)^{\frac{1}{\gamma}}}{(R-r)^{\frac{2}{\gamma}}}\|\overline{u}\|_{L^{{\frac{{\gamma}p}{p-n+1}}}(B_R)}.
\end{equation}
We iterate \eqref{itt} to get the desired estimate. Set
$$\gamma_i=2\chi^{i}\quad\text{and}\quad R_i=r+\frac{R-r}{2^{i}},\  i=0,1,2,\cdots,$$
i.e.,
$$\gamma_i=\chi\gamma_{i-1}\ \text{and} \ R_{i-1}-R_i=\frac{R-r}{2^{i}},\  i=1,2,\cdots.$$
By \eqref{itt},
$$\|\overline{u}\|_{L^{\frac{2p}{p-n+1}\chi^{i+1}}(B_{R_{i+1}})}
\leq C^{\sum_{j=0}^i\frac{1}{\gamma_j}}\cdot\prod_{j=0}^i \gamma_j^{\frac{1}{\gamma_j}}\cdot 4^{\sum_{j=0}^i\frac{j}{\gamma_j}}\frac{1}{(R-r)^{\sum_{j=0}^i\frac{2}{\gamma_j}}}\cdot\|\overline{u}\|_{L^{\frac{2p}{p-n+1}}(B_R)}.$$
Letting $i\to\infty$, by Young's inequality, we have
$$\begin{aligned}
\|\overline{u}\|_{L^\infty(B_r)}&\leq \frac{C}{(R-r)^{\frac{\chi}{\chi-1}}}\|\overline{u}\|_{L^{\frac{2p}{p-n+1}}(B_R)}\\[4pt]
&=\frac{C}{(R-r)^{\frac{\chi}{\chi-1}}}\|\overline{u}\|_{L^{1}(B_R)}^{\frac{p-n+1}{2p}}\cdot\|\overline{u}\|_{L^{\infty}(B_R)}^{\frac{p+n-1}{2p}}\\[4pt]
&\leq \frac{1}{2}\|\overline{u}\|_{L^{\infty}(B_R)}+\frac{C}{(R-r)^{\frac{\chi}{\chi-1}\cdot\frac{2p}{p-n+1}}}\|\overline{u}\|_{L^1(B_R)}.
\end{aligned}$$
Set $f(t)=\|\overline{u}\|_{L^\infty(B_t)}$ for $t\in (0,1].$ Then for any $0<r<R\leq 1$
$$f(r)\leq \frac{1}{2}f(R)+\frac{C}{(R-r)^{\frac{\chi}{\chi-1}\cdot\frac{2p}{p-n+1}}}\|\overline{u}\|_{L^1(B_1)}.$$
We apply  Lemma \ref{lem4.3} to get
$$f(r)\leq \frac{C}{(R-r)^{\frac{\chi}{\chi-1}\cdot\frac{2p}{p-n+1}}}\|\overline{u}\|_{L^1(B_1)}.$$
The lemma follows by choosing $r=\frac{1}{2}$ and $R=1$.
\end{proof}

Now we can use Lemma \ref{est} to obtain the interior estimates for  \eqref{4-eq}. For simplicity, we only consider the homogeneous equation.
\begin{proof}[{Proof of Theorem \ref{main-thm}}]
Denote $a_{ij}=U^{ij}$, $\lambda=\Delta u$ and $d=\det D^2u$.
Firstly, we consider the case $0\leq \theta<1$.
We apply Lemma \ref{est} to equation \eqref{4-eq-h}, which yields
$$\sup_{B_{1/2}}w\leq C\|w\|_{L^1(B_1)}\leq C\|w\|_{L^{\frac{q}{1-\theta}}(B_1)}.$$
Since $w=(\det D^2 u)^{-(1-\theta)}\in L^{\frac{q}{1-\theta}}(B_1),$ we know
$\sup_{B_{1/2}}\{(\det D^2 u)^{-1}\}\leq C$.
For the upper bound of the determinant, we set $w=\frac{1}{v}$. A direct calculation yields
$$w_i=-\frac{v_i}{v^2},\ \
w_{ij}=\frac{2v_iv_j}{v^3}-\frac{v_{ij}}{v^2}.$$
Then $v$ satisfies
$$0=U^{ij}w_{ij}=\frac{2}{v^3}U^{ij}v_iv_j-\frac{1}{v^2}U^{ij}v_{ij}\geq-\frac{1}{v^2}U^{ij}v_{ij},$$
i.e.,
$$-D_i(U^{ij}D_jv)\leq 0\quad \text{in}\,\,B_{1}.$$
Similarly, we use Lemma \ref{est} to obtain
$$\sup_{B_{1/2}}v\leq C\|v\|_{L^1(B_1)}\leq C\|v\|_{L^{\frac{p}{n(1-\theta)}}(B_1)},$$
which implies
$\sup_{B_{1/2}} \det D^2u\leq C$.

Next, we consider the case $\theta=1$. Write  $w=\log d$, where $d\in L^{\frac{p}{n}}(B_1)$. Then we have
$$
w_i=\frac{d_i}{d},\ \
w_{ij}=-\frac{d_id_j}{d^2}+\frac{d_{ij}}{d},$$
which yields
$$0=U^{ij}w_{ij}=-\frac{1}{d^2}U^{ij}d_id_j+\frac{1}{d}U^{ij}d_{ij}\leq\frac{1}{d}U^{ij}d_{ij},$$
i.e.,
$$-D_i(U^{ij}D_jd)\leq 0\quad \text{in}\,\,B_{1}.$$
Similarly, write $w=-\log z$ with $z=(\det D^2u)^{-1}\in L^{q}(B_1)$. We have
$$w_i=-\frac{z_i}{z},\ \ w_{ij}=\frac{z_iz_j}{z^2}-\frac{z_{ij}}{z},$$
which means
$$0=U^{ij}w_{ij}=\frac{1}{z^2}U^{ij}z_iz_j-\frac{1}{z}U^{ij}z_{ij}\geq-\frac{1}{z}U^{ij}z_{ij},$$
i.e.,
$$-D_i(U^{ij}D_jz)\leq 0\quad \text{in}\,\,B_{1}.$$
Then we use Lemma \ref{est} to obtain the bounds of $\det D^2u$.

Once we have the determinant estimates, all the interior estimates follow.
\end{proof}

\begin{rem}
To get the upper bound of $\det D^2u$, we only need $u\in W^{2,p}(B_1) $ with $p>\frac{n^2}{2}$. In fact, by taking $\det D^2 u=v^{\frac{1}{1-\theta}}\ (0\leq\theta<1)$,  where $v$ defined as the proof of Theorem \ref{main-thm},
condition \eqref{elli} in Lemma \ref{est} becomes
$$\frac{v^{\frac{1}{1-\theta}}}{\lambda(x)}\leq a_{ij}(x)\leq \lambda(x)^{n-1}$$
Then the upper bound can be obtained by similar arguments as in Lemma \ref{est} without assumption on $(\det D^2u)^{-1}$.
\end{rem}

\begin{rem}
We can also consider the inhomogeneous equation
\begin{equation}\label{non-hom-eq}
D_j(U^{ij}D_i w)=f\quad\,\,\text{in}\,\,B_1
\end{equation}
where $f\in L^{{p_0}}(B_1).$ For the lower bound of $\det D^2u$, if we assume ${p_0}\geq\frac{p}{n-1}$, we can apply Lemma \ref{est} to \eqref{non-hom-eq} directly for all $\theta\in [0,1)$. However, for the upper bound of $\det D^2u$, $v=\frac{1}{w}$ satisfies
\begin{equation}\label{non-hom-eq1}
-D_j(U^{ij}D_i v)\leq fv^2\quad\,\,\text{in}\,\,B_1.
\end{equation}
Even we assume ${p_0}=\infty$, we only know $fv\in L^{\frac{p}{n(1-\theta)}}(B_1)$. Then we can only
apply Lemma \ref{est} to the equation \eqref{non-hom-eq1} to get the upper bound of $\det D^2u$
when $\frac{1}{n}\leq\theta<1$. Hence, we have all the higher interior estimates for $\theta\in[\frac{1}{n}, 1)$. Note that the case $\theta=0$ (Abreu's equation) and $\theta=\frac{1}{n+2}$ (affine mean curvature equation) are not included.
\end{rem}

Finally, we prove a Liouville type theorem.
\begin{proof}[Proof of Corollary \ref{Liouville}]
For $u$ in any $B_R\subset\mathbb R^n$, define
$u_R(x)=\frac{1}{R^2}u(Rx)$. Tthen we know that $U_R^{ij}(w_{R})_{ij}=0$ in $B_1.$ Since $u$ satisfies \eqref{uni-bound}, we know $u_R$ satisfies
$$\int_{B_1} |D^2 u_R|^p+(\det D^2 u_R)^{-q}\,dx\leq C.$$
Applying Theorem \ref{main-thm} to $u_R$, we find that there exists a $C>0$ independent of $R$, such that $\|u_R\|_{C^{4,\alpha}(B_{1/2})}\leq C$. In particular, we know that $\|D^3u_R\|_{L^\infty(B_{1/2})}\leq C$.  Hence
$$\|D^3u(Rx)\|_{L^\infty(B_{\frac{1}{2}})}\leq \frac{C}{R},$$
i.e.,
$$\|D^3u(x)\|_{L^\infty(B_{\frac{R}{2}})}\leq \frac{C}{R}.$$
Let $R\to +\infty$, we have $D^3u\equiv 0,$ which means $u$ is a quadratic function.
\end{proof}

%\begin{rem}
%When $\theta\in [0, 1/4]$,  \eqref{detcond} can be removed, so the interior estimate implies the Bernstein theorem.  Hence, by Theorem \ref{int-est}, we obtain a new proof without using Caffarelli-Gutirrez's H\"older estimate for the linearized Monge-Amp\`ere equation.
%\end{rem}


\begin{thebibliography}{99999}


\bibitem [Ab]{Ab}
Abreu, M.,
K\"ahler geometry of toric varieties and extremal metrics.
{\it Int. J. Math.} {\bf 9} (1998), no. 6, 641-651.

\bibitem[Ca]{Ca}
 Caffarelli, L. A.,
 Interior $W^{2,p}$ estimates for solutions of the Monge-Amp\`ere equation.
 {\it  Ann. Math.} {\bf 131} (1990), no. 1, 135-150.

\bibitem [CG] {CG}
Caffarelli, L. A.; Guti\'errez, C. E.,
Properties of solutions of the linearized Monge-Amp\`ere equation.
{\it Amer. J. Math.} {\bf 119} (1997), no. 2, 423-465.

\bibitem[CW]{CW}
Chau, A.; Weinkove, B., Monge-Amp\`ere functionals and the second boundary value problem,
{\it Math. Res. Lett.} {\bf 22} (2015), no. 4, 1005-1022.


\bibitem [CC] {CC}
Chen, X.X.; Cheng, J.R.,
On the constant scalar curvature K\"ahler metrics, apriori estimates (I). A priori
estimates. \textit{J. Amer. Math. Soc.} \textbf{34} (2021), no. 4, 909–936.

\bibitem[CHLS]{CHLS}
Chen, B.;  Han, Q.; Li, A.-M.; Sheng, L.,
Interior estimates for the n-dimensional Abreu's equation.
{\it Adv. Math.} {\bf 251} (2014), 35-46.

\bibitem[Ch]{Ch}
Chern, S. S.,
Affine minimal hypersurfaces.
{\it Minimal submanifolds and geodesics} (Proc. Japan-United States Sem., Tokyo, 1977), pp. 17-30, North-Holland, Amsterdam-New York, 1979.


\bibitem[DFS]{DFS}
De Philippis, G.; Figalli, A.; Savin, O.,
A note on interior $ W^{2, 1+\varepsilon} $ estimates for the Monge-Amp\`ere equation.
{\it Math. Ann.} {\bf 357}(1) (2013), 11-22.

\bibitem[DS]{DS}
Daskalopoulos, P.; Savin, O.,
On Monge-Amp\`ere equations with homogeneous right-hand sides.
{\it Comm. Pure Appl. Math.} {\bf 62 }(2009), no. 5, 639-676.

%\bibitem[D1]{D1}
%Donaldson, S. K.
%Scalar curvature and stability of toric varieties.
%{\it J. Diff. Geom.}  {\bf 62}  (2002),  no. 2, 289-349.

\bibitem[D]{D}
Donaldson, S. K.,
Interior estimates for solutions of Abreu's equation.
{\it Collect. Math.}  {\bf 56}  (2005),  no. 2, 103-142.

%\bibitem[D2]{D3} Donaldson, S. K. Extremal metrics on toric surfaces: a continuity method.  {\it J. Diff. Geom. } {\bf 79}  (2008),  no. 3, 389-432.

%\bibitem[D4]{D4} Donaldson, S. K. Constant scalar curvature metrics on toric surfaces.  {\it Geom. Funct. Anal.}  {\bf 19}  (2009),  no. 1, 83-136.

\bibitem[F]{F}
Figalli, A.,
The Monge-Amp\`ere equation and its applications. {\it Zurich Lectures in Advanced Mathematics}. European Mathematical Society (EMS), Zürich, 2017.

\bibitem[GT]{GT}
Gilbarg, D.; Trudinger, N. S.,
{\em Elliptic partial differential equations of second order}.
Reprint of the 1998 edition. Classics in Mathematics. Springer-Verlag, Berlin, 2001.

\bibitem[GP]{GP}
Guan, P.; Phong, D. H.,
Partial Legendre transforms of non-linear equations.
{\it Proc. Amer. Math. Soc.} {\bf 140} (2012), no. 11, 3831-3842.

\bibitem[GN1]{GN1}
Guti\'errez, C. E.; Nguyen, T.,
Interior gradient estimates for solutions to the linearized Monge-Amp\`ere equation.
{\it Adv. Math.} {\bf 228} (2011), no. 4, 2034-2070.

\bibitem[GN2]{GN2} Guti\'errez, C. E.; Nguyen, T.,
Interior second derivative estimates for solutions to the linearized Monge-Amp\`ere
equation. {\it Trans. Amer. Math. Soc.} {\bf 367} (2015), no. 7, 4537-4568.

\bibitem[H]{H}
Heinz, H.,
\"Uber die Differential ungleichung $0<\alpha\leq rt-s^2\leq \beta<\infty$.
{\it Math. Z.} {\bf 72} (1959),107-126.

\bibitem[HL]{HL}
Han, Q.; Lin, F. H.,
{\em Elliptic partial differential equations, second edition.}
Courant Lecture Notes in Mathematics 1, Courant Institute of Mathematical Sciences, New York; American
Mathematical Society, Providence, RI, 2011.

\bibitem [JL]{JL}
Jia, F.; Li, A. M.,
A Bernstein property of some fourth order partial differential equations,
{\it Results Math.} {\bf 56} (2009), no. 1-4, 109-139.

\bibitem[Le1]{Le1}
Le, N. Q.,
Global second derivative estimates for the second boundary value problem of the prescribed affine mean curvature and Abreu's equations.
{\it Int. Math. Res. Not. IMRN} (2013), no. 11, 2421-2438.

\bibitem[Le2]{Le2}
Le, N. Q.,
$W^{4,p}$ solution to the second boundary value problem of the prescribed affine mean curvature and Abreu's equations.
{\it J. Diff. Eqn.} {\bf 260} (2016), no. 5, 4285-4300.

\bibitem[Le3]{Le3}
Le, N. Q.,
H\"older Regularity of the 2D Dual Semigeostrophic Equations via Analysis of Linearized Monge-Amp\`ere Equations.
{\it Comm. Math. Phy.} {\bf 360}(1) (2018), 271-305.

\bibitem[Le4]{Le4}
Le, N. Q.,
Singular Abreu equations and minimizers of convex functionals with a convexity constraint.
{\it Comm. Pure Appl. Math.} {\bf 73} (2020), no. 10, 2248-2283.

 \bibitem[Le5]{Le5}
Le, N. Q.,
On singular Abreu equations in higher dimensions,
 {\it J. d'Analyse Math.} {\bf 144} (2021), no. 1, 191-205.

\bibitem[LN]{LN}
Le, N.Q.; Nguyen, T.,
Global $W^{1,p}$-estimates for solutions to the linearized Monge-Amp\`re equations,
{\it J. Geom. Anal.} {\bf 27} (2017), no. 3 1751-1788.

\bibitem[LS]{LS}
Le, N. Q.; Savin, O.,
Boundary Regularity for Solutions to the Linearized Monge–Amp\`ere Equations.
{\it Arch. Ration. Mech. Anal.} {\bf 210} (2013), no. 3, 813-836.

\bibitem[LZ]{LZ}
 Le, N. Q.; Zhou, B.,
Solvability of a class of singular fourth order equations of Monge-Amp\`ere type.
{\it  Ann. PDE} {\bf 7} (2021), no. 2, Paper No. 13, 32 pp.

\bibitem[Li]{Li}
Liu, J. K.,
Interior $C^2$ estimate for Monge-Amp\`ere equations in dimension two.
{\it Proc. Amer. Math. Soc.} {\bf 149} (2021), no. 6, 2479-2486.

\bibitem[Lo]{Lo}
Loeper, G., On the regularity of the polar factorization for time dependent maps.
{\it Calc. Var. Part. Diff. Eqns.}  {\bf 22}(3) (2005), 343-374.

\bibitem[MS]{MS}
Murthy, M. R. V.; Stampacchia, G.,
Boundary value problems for some degenerate-elliptic operators.
{\it  Ann. Mat. Pura Appl.} {\bf 80} (1968), 1-122.

\bibitem[S]{S}
Schmidt, T.,
$W^{2,1+\epsilon}$-estimates for the Monge-Amp\`ere equation,
{\it Adv. Math.} {\bf 240} (2013), 672-689.

\bibitem[TW]{TW}
Tian, G. J.; Wang, X.-J.,
A class of Sobolev type inequalities,
{\it Meth. Appl. Anal.} {\bf 15}(2) (2008), 263-276.


\bibitem[Tr]{Tr}
Trudinger, N. S.,
Linear elliptic operators with measurable coefficients.
{\it Ann. Scuola Norm. Sup. Pisa Cl. Sci. (3)} {\bf 27} (1973), 265-308.

%\bibitem[Tr1]{Tr1}
%Trudinger, N. S.,
%On the regularity of generalized solutions of linear non-uniformly elliptic equations.
%{\it Arch. Ration. Mech. Anal.} {\bf 42} (1971) no. 1, 50-62.


\bibitem [TW1] {TW1}
Trudinger, N. S.; Wang, X.-J.,
The Bernstein problem for affine maximal hypersurfaces.
{\it Invent.  Math.} {\bf 140} (2000), no. 2, 399-422.

\bibitem [TW2] {TW2}
Trudinger, N. S.; Wang, X.-J.,
The affine plateau problem.
{\it J. Amer. Math. Soc.} {\bf 18} (2005), no. 2, 253-289.

%\bibitem [TW3] {TW3}
%Trudinger, N.S.; Wang, X.J.,
%Boundary regularity for the Monge-Amp\`ere and
%affine maximal surface equations.
%{\it Ann. Math. (2)} {\bf 167} (2008),  no. 3, 993-1028.

%\bibitem [TW4] {TW4}
%Trudinger, Neil S.; Wang, X.J.,
%Bernstein-J\"orgens theorem for a fourth order partial differential equation.
%{\it J. Part. Diff. Eqns.} {\bf 15} (2002), no. 1, 78-88.

%\bibitem[TW4]{TW4}
%Trudinger, N. S.; Wang, X. J. The Monge-Amp\`{e}re equation and its
%geometric applications.  {\it Handbook of geometric analysis.} No. 1,  467-524, {\bf Adv. Lect. Math. (ALM), 7}, Int. Press, Somerville, MA, 2008.

\bibitem[Z1]{Z1}
Zhou, B.,
The Bernstein theorem for a class of fourth order equations.
{\it Calc. Var. Part. Diff. Eqns.} {\bf 43} (2012), no. 1-2, 25-44.

\bibitem[Z2]{Z2}
Zhou, B.,
The first boundary value problem for Abreu's equation.
{\it Int. Math. Res. Not.}  (2012), no. 7, 1439-1484.


\end{thebibliography}
\end{document}